\documentclass[a4paper,11pt]{amsart}

\usepackage[utf8]{inputenc}		
\usepackage[T1]{fontenc}
\usepackage[english]{babel}

\usepackage{amsfonts}			
\usepackage{amsmath}
\usepackage{amssymb}
\usepackage{amsthm}

\usepackage{graphicx}			
\usepackage{tikz}
\usetikzlibrary{cd}

\usepackage{hyperref}			
\hypersetup{colorlinks=false, urlcolor=black, linkcolor=black}
\usepackage{cleveref}

\usepackage{enumitem}			

\usepackage{color}				
\usepackage{float}

\usepackage{thmtools}
\usepackage{thm-restate}

\usepackage{mathrsfs}

\newcommand{\Z}{\mathbb{Z}}						
\newcommand{\R}{\mathbb{R}}						
\newcommand{\C}{\mathbb{C}}						

\renewcommand{\S}{\mathbb{S}}					
\newcommand{\B}{\mathbb{B}}

\newcommand{\eps}{\varepsilon}					

\newcommand{\dd}								
{\mathop{}\!\mathrm{d}}						
\newcommand{\ddn}[1]							
{\mathop{}\!\mathrm{d^{#1}}}

\newcommand{\abs}[1]							
{\left| #1 \right|}
\newcommand{\smallabs}[1]						
{\lvert #1 \rvert}	
\newcommand{\norm}[1]							
{\left\lVert #1 \right\rVert}	
\newcommand{\smallnorm}[1]						
{\lVert #1 \rVert}						
\newcommand{\ip}[2]								
{\left< #1 , #2 \right>}

\DeclareMathOperator{\intr}{int}				
\DeclareMathOperator{\vol}{vol}					
\DeclareMathOperator{\spt}{spt}					
\DeclareMathOperator{\diam}{diam}				

\DeclareMathOperator{\capac}{Cap}	
\DeclareMathOperator{\adm}{Adm}		

\DeclareMathOperator{\dist}{dist}

\newcommand{\loc}{\mathrm{loc}}

\renewcommand{\phi}{\varphi}

\def\Xint#1{\mathchoice
	{\XXint\displaystyle\textstyle{#1}}%
	{\XXint\textstyle\scriptstyle{#1}}%
	{\XXint\scriptstyle\scriptscriptstyle{#1}}%
	{\XXint\scriptscriptstyle\scriptscriptstyle{#1}}%
	\!\int}
\def\XXint#1#2#3{{\setbox0=\hbox{$#1{#2#3}{\int}$}
		\vcenter{\hbox{$#2#3$}}\kern-.5\wd0}}

\def\dashint{\Xint-}

\newtheorem{thm}{Theorem}[section]{\bf}{\it}
\newtheorem{lemma}[thm]{Lemma}
\newtheorem{prop}[thm]{Proposition}
\newtheorem{cor}[thm]{Corollary}

{\bf}{\it}

\newtheorem{innercustomthm}{Theorem}[section]{\bf}{\it}
\newenvironment{customthm}[1]
{\renewcommand\theinnercustomthm{#1}\innercustomthm}
{\endinnercustomthm}

{\bf}{\it}

{\bf}{\it}

\theoremstyle{definition}
\newtheorem{defn}[thm]{Definition}

\theoremstyle{remark}

\numberwithin{equation}{section}

\begin{document}
	
	\title{Linear distortion and rescaling for quasiregular values}

	\author[I. Kangasniemi]{Ilmari Kangasniemi}
	\address{Department of Mathematical Sciences, University of Cincinnati, P.O.\ Box 210025, 
			Cincinnati, OH 45221, USA.}
	\email{kangaski@ucmail.uc.edu}
	
	\author[J. Onninen]{Jani Onninen}
	\address{Department of Mathematics, Syracuse University, Syracuse,
		NY 13244, USA and  Department of Mathematics and Statistics, P.O.Box 35 (MaD) FI-40014 University of Jyv\"askyl\"a, Finland
	}
	\email{jkonnine@syr.edu}
	
	\subjclass[2020]{Primary 30C65; Secondary 35R45}
	\date{\today}
	\keywords{Quasiregular values, linear distortion, rescaling principle}
	\thanks{I. Kangasniemi was supported by the National Science Foundation grant DMS-2247469. J. Onninen was supported by the National Science Foundation grant DMS-2154943. }
	
\begin{abstract}
		Sobolev mappings exhibiting only pointwise quasiregularity-type bounds have arisen in various applications, leading to a recently developed theory of quasiregular values. In this article, we show that by using rescaling, one obtains a direct bridge between this theory and the classical theory of quasiregular maps. 
		More precisely,  we  prove  that a non-constant mapping $f \colon \Omega \to \mathbb{R}^n$ with a $(K, \Sigma)$-quasiregular value at $f(x_0)$ can be rescaled at $x_0$ to a non-constant $K$-quasiregular mapping.
		Our proof of this fact involves  establishing a quasiregular values -version of the linear distortion bound of quasiregular mappings. A quasiregular values variant of the small $K$ -theorem is obtained as an immediate  corollary of our main result.
	\end{abstract}
	
	\maketitle
	
	\section{Introduction}
	
	Let $\Omega$ be a domain in $\R^n$ with $n \ge 2$ and $y_0 \in \R^n$.  A mapping $f \colon \Omega \to \R^n$ in the Sobolev class $W^{1,n}_\loc(\Omega, \R^n)$ has a \emph{$(K, \Sigma)$-quasiregular value at $y_0$} if it satisfies the inequality 
	\begin{equation}\label{eq:qrval_def}
		\abs{Df(x)}^n \leq K J_f(x) + \abs{f(x)-y_0}^n \Sigma(x)
	\end{equation}
	for almost every (a.e.) $x \in \Omega$, where $K \geq 1$ is a constant and $\Sigma$ is a nonnegative function in $L^{1+\eps}_\loc(\Omega)$ for some $\eps >0$.  Here, $\abs{Df(x)}$ denotes the operator norm of the weak derivative of $f$ at $x$, and $J_f = \det Df$ is the Jacobian determinant of $f$.
 	The higher integrability of $\Sigma$ yields that a map $f \in W^{1,n}_\loc(\Omega, \R^n)$ with a $(K, \Sigma)$-quasiregular value has a continuous representative; see \cite[Theorem 1.1]{Kangasniemi-Onninen_Heterogeneous} or \cite{Dolezalova-Kangasniemi-Onninen_MGFD-cont} for details. In this paper, we always refer to the continuous representatives of these maps unless otherwise noted.
	
 	The classical analytic definition of a \emph{$K$-quasiregular} map is exactly the case $\Sigma \equiv 0$ of~\eqref{eq:qrval_def}. Of particular note is the sub-class of homeomorphic quasiregular maps, which are called \emph{quasiconformal maps}. The concept of quasiregular mappings emerged as a necessity to extend the geometric principles of holomorphic functions of one complex variable into higher dimensions, and is by now a central topic in modern analysis. Undoubtedly, the theory of arbitrary solutions of~\eqref{eq:qrval_def} is significantly more complicated.  Heuristically speaking,~\eqref{eq:qrval_def} allows $f$ to deviate from the standard behavior of quasiregular maps in a manner controlled by $\Sigma$, where these deviations  incur greater penalties as $f$ approaches $y_0$.  Away from $y_0$, these solutions behave similarly to an arbitrary Sobolev map. 
 
  	In the planar case, the equation~\eqref{eq:qrval_def} can be expressed as a  linear  Cauchy-Riemann type system, and is thus uniformly elliptic. If $n = 2$, then a solution $f$ can be decomposed using the existence theory of Beltrami equations into
 	\begin{equation}\label{eq:2d_decomposition}
  		f(x) = e^{\theta(x)} g(x) + f(x_0)
  	\end{equation}
   	with $\theta \colon \Omega \to \C$ continuous and $g \colon \Omega \to \C$ quasiregular, see~\cite{Astala-Iwaniec-Martin_Book}.
  	A notable application of this class of mappings arose when Astala and P\"aiv\"arinta employed it in their solution of the planar Calder\'on problem \cite{Astala-Paivarinta}. Moreover, solutions derived from \eqref{eq:qrval_def} also serve as pivotal components in various other uniqueness theorems; see the book by Astala, Iwaniec and Martin~\cite{Astala-Iwaniec-Martin_Book} for further details. When $n \ge 3$, the theory is instead nonlinear, and the main inherent difficulty lies in the lack of general existence theorems.

 	Despite the challenges inherent in the higher dimensional version of the theory, we have recently managed to show that mappings with quasiregular values satisfy single-value counterparts to many of the foundational results of quasiregular maps, under suitable regularity assumptions on $\Sigma$. These results include versions of Reshetnyak's open mapping theorem \cite{Kangasniemi-Onninen_1ptReshetnyak},  the Liouville uniqueness theorem \cite{Kangasniemi-Onninen_Heterogeneous, Kangasniemi-Onninen_Heterogeneous-corrigendum}, and even Rickman's Picard theorem \cite{Kangasniemi-Onninen_Picard} on the number omitted values of entire quasiregular maps. By our variant of Reshetnyak's theorem, if a non-constant map $f$ has a quasiregular value at $y_0$, then $f^{-1} \{y_0\}$ is a discrete set and the local index $i(x, f)$ is positive at every $x \in f^{-1}\{y_0\}$. Nevertheless, it's crucial to understand that a map $f$ satisfying~\eqref{eq:qrval_def} needs not be locally quasiregular even in any neighborhood of~$f^{-1} \{y_0\}$. In fact, it is possible that such a neighborhood always meets a region where $J_f <0$, see~\cite[Example 7.2]{Kangasniemi-Onninen_1ptReshetnyak}.
	
	Our main result in this work is a version of the rescaling principle for quasiregular values. In particular, we show that if $f \colon \Omega \to \R^n$ has a $(K, \Sigma)$-quasiregular value at $f(x_0)$ with $\Sigma \in L^{1+\eps}_\loc(\Omega)$, then $f$ can be rescaled into a $K$-quasiregular map. For the statement, we recall that if $h, h_j \in X$ where $X$ is a Banach space, we say that \emph{$h_j$ converges to $h$ weakly in $X$} if $\Phi(h_j) \to \Phi(h)$ for every bounded functional $\Phi \in X^*$. Also, we denote the open unit ball centered at the origin in $\R^n$ by $\B^n$.
	
	\begin{thm}\label{thm:rescaling}
		Let $\Omega \subset \R^n$ be a domain and  $x_0 \in \Omega$. Suppose that $f \colon \Omega \to \R^n$ is a non-constant continuous map having a $(K, \Sigma)$-quasiregular value at $f(x_0)$, where $K \geq 1$ and $\Sigma \in L^{1+\eps}_\loc(\Omega)$ with $\eps > 0$. Then there exist radii $r_j > 0$ and scaling factors $c_j > 0$, $j \in \Z_{> 0}$, such that the functions $h_j \colon \B^n \to \R^n$ defined by
		\[
			h_j(x) = c_j(f(x_0 + r_j x) - f(x_0))
		\]
		converge both locally uniformly and weakly in $W^{1,n}(\B^n, \R^n)$ to a non-constant $K$-quasiregular map $h \colon \B^n \to \R^n$.
	\end{thm}

	This result is again relatively easy to see in the planar case, thanks to the planar decomposition~\eqref{eq:2d_decomposition}. Indeed, an appropriate rescaling will hence make the $e^{\theta}$-coefficient approach a constant value by continuity of $\theta$, allowing one to reduce the result to a similar rescaling principle for quasiregular maps. However, as is common in the theory of quasiregular values, the question is much more involved in higher dimensions due to a lack of a counterpart for this decomposition.
	
	\subsection{Linear distortion}
	
	Our main tool in achieving Theorem \ref{thm:rescaling} is a linear distortion bound for mappings with quasiregular values. That is, given a continuous map $f \colon \Omega \to \R^n$, a point $x_0 \in \R^n$, and a radius $r > 0$ such that $\overline{\B^n}(x_0, r) \subset \Omega$, we let
	\begin{equation}\label{eq:L_and_l}
		\begin{aligned}
		L_f(x_0, r) &= \sup \left\{ \abs{y - f(x_0)} : y \in f \B^n(x_0, r)\right\},\\
		l_f(x_0, r) &= \inf \left\{ \abs{y - f(x_0)} : y \notin f \B^n(x_0, r)\right\}.
		\end{aligned}
	\end{equation}
	If $f$ is a $K$-quasiregular map, then $\limsup_{r \to 0} L_f(x_0, r)/l_f(x_0, r)$ is bounded from above by a constant only dependent on $n$, $K$, and the local index $i(x_0, f)$. An explicit bound for quasiconformal maps was given by Gehring~\cite{Gehring_TAMS} and for quasiregular maps by Martio, Rickman and V\"ais\"al\"a~\cite[Theorem~4.5]{Martio-Rickman-Vaisala_AASF-QR1}.  We prove a counterpart of this result for quasiregular values.
	
	\begin{thm}\label{thm:limsup_lin_distortion}
		Let $\Omega \subset \R^n$ be a domain, let $x_0 \in \Omega$, and let $f \in W^{1,n}_\loc(\Omega, \R^n)$ be a non-constant continuous map. Suppose that $f$ has a $(K, \Sigma)$-quasiregular value at $f(x_0)$, where $K \geq 1$ and $\Sigma \in L^{1+\eps}_\loc(\Omega)$ with $\eps > 0$. Then
		\[
			\limsup_{r \to 0} \frac{L_f(x_0, r)}{l_f(x_0, r)} \leq C(n, K, i(x_0, f)) < \infty.
		\]
	\end{thm}

	The main difficulty in proving Theorem \ref{thm:limsup_lin_distortion} is caused by the fact that, as opposed to the theory of quasiregular maps, a map with just a quasiregular value at $f(x_0)$ can have bounded components in the complement of a standard normal neighborhood of $x_0$. The hard part is hence obtaining control on these bounded components, which we achieve in Section \ref{sect:bounded_components} through a somewhat intricate argument combining degree theory, the Loewner property of $\R^n$, Gehring's Lemma, and the "pseudo-logarithm" $\R^n \setminus \{0\} \to \R \times \S^{n-1}$ considered also in e.g.\ \cite{Kangasniemi-Onninen_Heterogeneous, Kangasniemi-Onninen_Heterogeneous-corrigendum, Kangasniemi-Onninen_Picard}.
	
	\subsection{Small $K$ -theorem.}
	
	As a near-immediate consequence of Theorem \ref{thm:rescaling}, we obtain a small $K$ -theorem for maps with quasiregular values. 
	
	We recall that the small $K$ -theorem of quasiregular maps states that there exists a constant $K_0 = K_0(n) > 1$ such that if $n \geq 3$ and $f \colon \Omega \to \R^n$ is a $K$-quasiregular map with $K < K_0$, then $f$ is a local homeomorphism. The existence of such a $K_0$ can be most easily shown by indirect means, via normal family methods: see e.g.\ \cite{Iwaniec_abstract-small-K}. It is conjectured that $K_0 = 2^{n-1}$, with the winding map as the extremal example. The current best result, which provides a concrete lower bound on $K_0$ that is regardless very far from $2^{n-1}$, is due to Rajala \cite[Theorem 1.1]{Rajala_smallK}. We call this $K_0$ \emph{Rajala's constant}. The very closely related corresponding conjecture for the inner distortion of quasiregular maps is known as the \emph{Martio conjecture}, and is one of the major open problems of the theory of quasiregular maps; see e.g.\ \cite{Tengvall_Martio-conj-remarks} for details.
	
	Our version of the small $K$ -theorem for quasiregular values is as follows.
	
	\begin{thm}\label{thm:small_K}
		Let $\Omega \subset \R^n$ be a domain, let $x_0 \in \Omega$, and let $f \in W^{1,n}_\loc(\Omega, \R^n)$ be a non-constant continuous map. Suppose that $f$ has a $(K, \Sigma)$-quasiregular value at $f(x_0)$, where $\Sigma \in L^{1+\eps}_\loc(\Omega)$ with $\eps > 0$. If $K < K_0(n)$, where $K_0(n)$ is Rajala's constant, then $i(x_0, f) = 1$.
	\end{thm}
	
	Note that we do \emph{not} obtain that $f$ is a homeomorphism in some neighborhood of $x_0$; indeed, as already pointed out, such a result is not possible by \cite[Example 7.2]{Kangasniemi-Onninen_1ptReshetnyak}. For quasiregular values, the small $K$ -theorem hence only provides an infinitesimal form of injectivity, not a local form of injectivity.
	
	\subsection*{Acknowledgments}
	
	The authors thank  Nageswari Shanmugalingam for several discussions that helped in improving the paper.

	\section{Preliminaries on classical topics}
	
	\subsection{Degree and local index} 
	
	If $U$ is open, $\overline{U}$ is a compact subset of $\Omega$, and $f \colon \overline{U} \to \R^n$ is continuous, then the \emph{topological degree} $\deg(f, y, U) \in \Z$ is well-defined for all $y \in \R^n \setminus f(\partial U)$. Since the full definition of the topological degree is relatively technical, we elect not to state it here; see e.g.\ \cite[Chapter~1]{Fonseca-Gangbo-book} for the details. Instead, we recall the basic properties of the topological degree that we use; for proofs, see e.g.\ \cite[Theorems 2.1, 2.3 (3), 2.7]{Fonseca-Gangbo-book}.
	
	\begin{lemma}\label{lem:top_degree_props}
		Let $\Omega \subset \R^n$ be an open set, let $f \in C(\Omega, \R^n)$, and let $U \subset \Omega$ be open and compactly contained in $\Omega$. The topological degree $\deg(f, \cdot, U)$ satisfies the following properties. 
		\begin{enumerate}
			\item \label{enum:deg_constant} (Local constancy) The mapping $y \mapsto \deg(f, y, U)$ is locally constant in $\R^n \setminus f \partial U$.
			\item \label{enum:deg_image} (Vanishing outside the image set) If $y \in \R^n \setminus f \partial U$ and $\deg(f, y, U) \neq 0$, then $y \in f U$.
			\item \label{enum:deg_addition} (Additivity) If $U = \bigcup_{i \in I} U_i$, where $I$ is at most countable, $U_i$ are mutually disjoint open sets,  and $y \in \R^n \setminus \left( \bigcup_{i \in I} \partial U_i \right)$, then
			\[
				\deg(f, y, U) = \sum_{i \in I} \deg(f, y, U_i).
			\]
			\item \label{enum:deg_excision} (Excision) If $K \subset \overline{U}$ is compact and $y \in \R^n \setminus f (K \cup \partial U)$, then
			\[
				\deg(f, y, U) = \deg(f, y, U \setminus K).
			\]
			\item \label{enum:deg_homotopy} (Homotopy invariance) If $g \in C(\Omega, \R^n)$, $H \colon I \times \overline{U} \to \R^n$ is a homotopy from $f$ to $g$, and $y \in \R^n \setminus H(I \times \partial U)$, then
			\[
				\deg(f, y, U) = \deg(g, y, U).
			\]
		\end{enumerate}
	\end{lemma}
	
	Moreover, continuous Sobolev mappings satisfy a change of variables theorem in terms of the topological degree. The following version of this result follows from e.g.\ the discussion in \cite[Theorems 5.27, 5.21, and Remark 5.29]{Fonseca-Gangbo-book} with relatively straightforward adjustments to the arguments.
	\begin{lemma}\label{lem:change_of_vars}
		Let $\Omega \subset \R^n$ be an open set, let $v \in L^\infty(\R^n)$, and let $U \subset \Omega$ be open and compactly contained in $\Omega$. Suppose that $f \in C(\Omega, \R^n) \cap W^{1,p}_\loc(\Omega, \R^n)$, where either $p > n$, or $p = n$ and $f$ satisfies the Lusin (N) -property. Then $v \cdot \deg(f, \cdot, U) \in L^1(\R^n \setminus f \partial U)$, $v \circ f \cdot J_f \in L^1_\loc(\Omega)$ despite the fact that $v \circ f$ is not even necessarily measurable, and we moreover have
		\[
			\int_{U} v \circ f \cdot J_f = \int_{\R^n \setminus f \partial U} v \cdot \deg(f, \cdot, U).
		\]
	\end{lemma}

	Suppose then that $\Omega \subset \R^n$ is open, $f \colon \Omega \to \R^n$ is continuous, and $x_0 \in \Omega$ is such that $f^{-1}\{f(x_0)\}$ is a discrete subset of $\Omega$. Then there exists a neighborhood $U$ of $x_0$ compactly contained in $\Omega$ such that $\overline{U} \cap f^{-1}\{f(x_0)\} = \{x_0\}$. By the excision property of the degree, see Lemma \ref{lem:top_degree_props} \eqref{enum:deg_excision}, the value of $\deg(f, f(x_0), U)$ is independent on the choice of such $U$; this value is called the \emph{local index of $f$ at $x_0$}, and is denoted $i(x_0, f)$.
	
	We also recall that the local index satisfies the following summation formula. For a proof, see e.g.\ \cite[Theorem 2.9 (1)]{Fonseca-Gangbo-book}.
	
	\begin{lemma}\label{lem:local_index_summation}
		Let $\Omega \subset \R^n$ be open, let $f \colon \Omega \to \R^n$ be continuous, and let $y \in \R^n$ be a point for which $f^{-1}\{y\}$ is discrete. If $U$ is a bounded domain with $\overline{U} \subset \Omega$ and $y \notin f \partial U$, then
		\[
		\deg(f, y, U) = \sum_{x \in U \cap f^{-1} \{y\}} i(x, f).
		\]
	\end{lemma}

	\subsection{Capacity}
	
	A \emph{condenser} in $\R^n$ is a pair $(C, U)$, where $C \subset U \subset \R^n$, $C$ is compact, and $U$ is open. If $p \in [1, \infty)$, we say that a function $u \in W^{1,p}_\loc(\R^n) \cap C(\R^n)$ is \emph{admissible} for the a condenser $(C, U)$ if $u \geq 1$ on $C$ and $u \leq 0$ on $\R^n \setminus U$. The family of admissible functions for $(C, U)$ is denoted $\adm(C, U)$, and the \emph{$p$-capacity} of a condenser $(C, U)$ is defined by
	\[
		\capac_p(C, U) = \inf_{u \in \adm(C, U)} \int_{\R^n} \abs{\nabla u}^p \vol_n.
	\]
	For the specific value $p = n$, the capacity $\capac_n(C, U)$ is known as the \emph{conformal capacity} of $(C, U)$, and will be abbreviated as $\capac(C, U)$.
	
	We note that $\capac_p$ is monotone in the following sense: if $C_1 \subset C_2$ and $U_1 \supset U_2$, then
	\begin{equation}\label{eq:capac_monotonicity}
		\capac_p(C_1, U_1) \leq \capac_p(C_2, U_2).
	\end{equation}
	Indeed, this is an immediate consequence of the definition as a function admissible for $(C_2, U_2)$ is also admissible for $(C_1, U_1)$.
	
	We single out the following consequence of the definition of $p$-capacity that we use numerous times in this article.
	
	\begin{lemma}\label{lem:basic_adm_function_estimate}
		Let $p \in [1, \infty)$, let $a, b \in (0, \infty)$ with $a < b$, and let $C \subset U \subset \R^n$, where $C$ is compact and $U$ is a bounded open set. Suppose that $f \in C(\overline{U} \setminus \intr C, \R^n \setminus \{0\})$ is such that $f\vert_{U \setminus C} \in W^{1,p}(U \setminus C, \R^n)$, $\abs{f} \leq a$ on $\partial U$, and $\abs{f} \geq b$ on $\partial C$. Then
		\[
			\capac_p(C, U) \leq \frac{1}{\log^p(b/a)} \int_{U \setminus C} \frac{\abs{Df}^p}{\abs{f}^p}.
		\] 
	\end{lemma}
	\begin{proof}
		Let $\eps \in (0, (b-a)/2)$. We adjust the condenser slightly by selecting an open $U_\eps$ and a compact $C_\eps$ for which $C \subset \intr C_\eps$, $C_\eps \subset U_\eps$, $\overline{U_\eps} \subset U$, $\abs{f} \leq a + \eps$ on $C_\eps \setminus \intr C$, and $\abs{f} \geq b - \eps$ on $\overline{U} \setminus U_\eps$. We then define a function
		\[
			u_\eps = \min \left(1, \max \left( 0,  \frac{\log \abs{f} - \log (a + \eps)}{\log (b - \eps) - \log (a + \eps)}\right) \right).
		\]
		Since $0$ is not in the image of $f$, $u_\eps$ is a well-defined continuous map in $C(U \setminus C)$. Moreover, by boundedness of $U$, we see that $\abs{f}$ has a positive minimum on $\overline{U} \setminus \intr C$. Thus, by a Sobolev chain rule such as the one shown in \cite{Ambrosio-DalMaso}, and since Sobolev spaces are closed under minima and maxima as discussed e.g.\ in \cite[Theorem 1.20]{Heinonen-Kilpelainen-Martio_book}, we have $u \in W^{1,p}(U\setminus C)$ with
		\[
			\abs{\nabla u_\eps} \leq \frac{\nabla \abs{f}}{\abs{f} \log \frac{b-\eps}{a+\eps}} \leq \frac{\abs{Df}}{\abs{f} \log \frac{b-\eps}{a+\eps}}.
		\]
		
		We also have $u_\eps \equiv 1$ on $C_\eps \setminus \intr C$ and $u_\eps \equiv 0$ on $\overline{U} \setminus U_\eps$. Thus, we can extend $u_\eps$ to $\R^n$ by defining $u_\eps \equiv 1$ on $\intr C_\eps$ and $u_\eps \equiv 0$ on $\R^n \setminus \overline{U_\eps}$. As these domains of definition are open and cover all of $\R^n$, we conclude that $u_\eps \in C(\R^n) \cap W^{1,p}_\loc(\R^n)$. Thus, $u_\eps$ is admissible for $(C_\eps, U_\eps)$. We can then use \eqref{eq:capac_monotonicity}, the gradient estimate for $u_\eps$, and the fact that $\nabla u_\eps = 0$ outside $U \setminus C$, in order to obtain
		\[
			\capac_p(C, U) \leq \capac_p(C_\eps, U_\eps) \leq \int_{\R^n} \abs{\nabla u_\eps}^p \leq \frac{1}{\log^p \frac{b-\eps}{a+\eps}} \int_{U \setminus C} \frac{\abs{Df}^p}{\abs{f}^p}.
		\]
		Letting $\eps \to 0$, the claim follows.
	\end{proof}
	
	We then recall the \emph{Loewner property} of $\R^n$; see e.g.\ \cite[Section 3]{Heinonen-Koskela_Acta} for details.
	\begin{lemma}\label{lem:Loewner}
		Let $(C, U)$ be a condenser in $\R^n$ with $n \geq 2$. Suppose that $C$ and $\R^n \setminus U$ are connected sets with more than one point, and that $C$ is bounded. Then
		\[
			\capac(C, U) \geq \phi_n \left( \frac{\dist(C, \R^n \setminus U)}{\min(\diam C, \diam (\R^n \setminus U))} \right),
		\]
		where $\phi_n \colon (0, \infty) \to (0, \infty)$ is a decreasing function depending only on $n$ with
		\[
			\phi_n(t) \geq C(n) \max\left(\log \frac{1}{t}, \frac{1}{\log^{n-1}(t)}\right)
		\]
	\end{lemma}
	
	We point out the following special case of the Loewner property; see also \cite[Theorem~11.7 (4)]{Vaisala_book}.
	
	\begin{lemma}\label{lem:ring_condenser_capacity}
		Let $C_1, C_2$ be closed, connected sets with $C_1$ compact. Suppose that $x_0 \in C_1$, that $C_2$ is unbounded, and that both $C_1$ and $C_2$ meet $\partial \B^n(x_0, R)$ for some $R > 0$. Then
		\[
		\capac(C_1, \R^n \setminus C_2) \geq C(n) > 0.
		\]
	\end{lemma}
	\begin{proof}
		The claim follows from Lemma \ref{lem:Loewner}, since we have $\diam(C_2) = \infty$ and $\dist(C_1, C_2)/\diam(C_1) \leq R/R = 1$.
	\end{proof}
	
	We also need the following simple lower bound for the $p$-capacity of a point inside a cube when $p > n$.
	
	\begin{lemma}\label{lem:point_capacity}
		Let $p > n$, let $Q$ be an open cube in $\R^n$, and let $x_0 \in Q$. Then
		\[
			\capac_p(\{x_0\}, Q) \geq C(n, p) \abs{Q}^{-\frac{p-n}{n}}.
		\]
	\end{lemma}
	\begin{proof}
		Let $s$ be the length of the side of $Q$, and let $r$ be the smallest radius for which $Q \subset \B^n(x_0, r)$. Note that $r < \sqrt{n}s$ and $s = \abs{Q}^{1/n}$. Now, by the monotonicity of capacity \eqref{eq:capac_monotonicity} and the standard capacity formula for $(\{x_0\}, \B^n(x_0, r))$ given e.g.\ in \cite[Example 2.12]{Heinonen-Kilpelainen-Martio_book}, we have
		\begin{multline*}
			\capac_p(\{x_0\}, Q) \geq  \capac_p(\{x_0\}, \B^n(x_0, r))\\ 
			= \frac{C(n, p)}{r^{p-n}} > \frac{n^\frac{p-n}{2}C(n,p)}{s^{p-n}} 
			= n^\frac{p-n}{2} C(n,p) \abs{Q}^{-\frac{p-n}{n}},
		\end{multline*}
		which completes the proof.
	\end{proof}
	
	\subsection{Gehring's lemma}
	
	We also recall a local version of Gehring's lemma that we use in our arguments. The following version is given by Iwaniec in \cite{Iwaniec_GehringLemma}. Note that if $Q$ is a cube in $\R^n$ and $c > 0$, we use $cQ$ to denote the cube $Q$ scaled by $c$ with the same center.
	
	\begin{prop}[{\cite[Proposition 6.1]{Iwaniec_GehringLemma}}]
		\label{prop:local_Gehring_lemma}
		Let $Q_0$ be a cube in $\R^n$, and let $g, h \in L^p(Q_0)$, $1 < p < \infty$, be non-negative functions satisfying
		\[
		\left( \dashint_Q g^p \right)^\frac{1}{p} \leq C_0 \dashint_{2Q} g + \left( \dashint_{2Q} h^p \right)^\frac{1}{p}
		\]
		for all cubes $Q$ with $2Q \subset Q_0$. Then there exists $q_0 = q_0(n, p, C_0) > p$ such that for all $q \in (p, q_0)$ and $\sigma \in (0, 1)$, we have
		\[
		\left( \dashint_{\sigma Q_0} g^q \right)^\frac{1}{q} \leq C(n, p, q, \sigma) \left(  \left( \dashint_{Q_0} g^p \right)^\frac{1}{p} + \left( \dashint_{Q_0} h^q \right)^\frac{1}{q} \right).
		\]
	\end{prop}
	
	\section{Normal neighborhoods and linear dilatations for non-open mappings}
	
	We then outline some basic theory of normal neighborhoods and linear dilatations for continuous mappings $f \colon \Omega \to \R^n$. For open discrete mappings, this theory is standard; see e.g.\ \cite[Chapter I.4]{Rickman_book}. However, in our setting we only have openness and discreteness at a single value, which causes some differences to the theory.
	
	\subsection{Normal neighborhoods}
	
	We begin by introducing normal neighborhoods. Due to differences between our setting and the setting of open discrete maps, we need an additional stronger level of the concept.
	
	\begin{defn}\label{def:normal_neighborhoods}
		Let $\Omega \subset \R^n$ be open, let $f \colon \Omega \to \R^n$ be continuous, and let $x_0 \in \Omega$ be a point for which $f^{-1}\{f(x_0)\}$ is discrete. We say that a bounded open set $U \subset \R^n$ is a \emph{normal neighborhood} of $x_0$ (under $f$) if $\overline{U} \subset \Omega$, $x_0 \in U$, $\overline{U} \cap f^{-1}\{f(x_0)\} = \{x_0\}$, $f(x_0) \in \intr f U$, and $f\partial U \subset \partial f U$. Moreover, we say that $U \subset \R^n$ is a \emph{strong normal neighborhood} of $x_0$ (under $f$) if $U$ is a normal neighborhood of $x_0$ and there exists a larger normal neighborhood $U_0$ of $x_0$ such that $U \subset U_0$ and every bounded component of $\R^n \setminus U$ is contained in $U_0$.
	\end{defn}
	
	We note the following immediate consequences of the definitions of normal and strong normal neighborhoods.
	
	\begin{lemma}\label{lem:normal_nbhd_props}
		Let $\Omega \subset \R^n$ be open, let $f \colon \Omega \to \R^n$ be continuous, and let $x_0 \in \Omega$ be a point for which $f^{-1}\{f(x_0)\}$ is discrete. If $U$ is a normal neighborhood of $x_0$, then 
		\begin{equation}\label{eq:normal_neighborhood_degree}
			\deg(f, f(x_0), U) = i(x_0, f).
		\end{equation}
		Moreover, if $U$ is a strong normal neighborhood of $x_0$ and $V$ is a bounded component of $\R^n \setminus U$, then $f^{-1}\{f(x_0)\} \cap V = \emptyset$.
	\end{lemma}
	\begin{proof}
		$f(x_0) \in \intr f U$ and $f\partial U \subset \partial f U$ imply that $\deg(f, f(x_0), U)$ is well defined, and Lemma \ref{lem:local_index_summation} yields \eqref{eq:normal_neighborhood_degree}. The claim that $f^{-1}\{f(x_0)\} \cap V = \emptyset$ follows from the fact that $V \subset U_0$ for some larger neighborhood $U_0$, $f^{-1}\{f(x_0)\} \cap U_0 = \{x_0\}$, and $x_0 \in U \subset \R^n \setminus V$.
	\end{proof}
	
	If $\Omega \subset \R^n$ is open and $f \colon \Omega \to \R^n$ is continuous, then for every point $x_0 \in \Omega$ and every $\rho > 0$, we use $U_f(x_0, \rho)$ to denote the connected component of $f^{-1} \B^n(f(x_0), \rho)$ that contains $x_0$. Since $\R^n$ is locally connected, connected components of open subsets of $\R^n$ are open. Thus, the sets $U_f(x_0, \rho)$ are open neighborhoods of $x_0$. 
	
	We then point out the following basic topological lemma.
	
	\begin{lemma}\label{lem:boundary_into_boundary}
		Let $\Omega \subset \R^n$ be open, let $x_0 \in \Omega$, and let $f \colon \Omega \to \R^n$ be continuous. If $V \subset \R^n$ is open and $U \subset \Omega$ is a connected component of $f^{-1} V$, then $f \partial U \subset \partial V$. In particular, for every $\rho > 0$, we have $f \partial U_f(x_0, \rho) \subset \partial \B^n(f(x_0), \rho)$.
	\end{lemma}
	\begin{proof}
		Suppose $x \in \partial U$, and let $B$ be a neighborhood of $f(x)$. By continuity of $f$, there exists a neighborhood $W$ of $x$ with $f(W) \subset B$. Since $x \in \partial U$, $B$ meets $V$ at a point $z$. Thus, since $U \subset f^{-1}V$, $B$ meets $V$ at $f(z)$. As this holds for all neighborhoods $B$ of $f(x)$, we conclude that $f(x) \in \overline{V}$.
		
		Suppose then towards contradiction that $f(x) \in V$. Then $\{x\} \cup U \subset f^{-1} V$. Moreover, since $U$ is connected and $x \in \partial U$, we have that $\{x\} \cup U$ is connected. Moreover, since $\R^n$ is locally connected, and $U$ is a component of an open subset of $\R^n$, $U$ is open, and therefore $x \notin U$. Thus, $\{x\} \cup U$ is a connected subset of $f^{-1} V$ larger than the connected component $U$ of $f^{-1} V$. This is a contradiction, so we must have $f(x) \notin V$, and therefore $f(x) \in \partial V$.
	\end{proof}
	
	By Lemma \ref{lem:boundary_into_boundary}, the degree $\deg (f, y, U_f(x_0, \rho))$ is well-defined for all $y \notin \partial \B^n(f(x_0), \rho)$ whenever $U_f(x_0, \rho)$ is compactly contained in $\Omega$.  We then recall a standard surjectivity result for the sets $U_f(x_0, \rho)$ in cases where the degree does not vanish.
	
	\begin{lemma}\label{lem:Uf_surj}
		Let $\Omega \subset \R^n$ be open, let $x_0 \in \Omega$, and let $f \colon \Omega \to \R^n$ be continuous. Suppose that $\rho > 0$ is such that $\overline{U_f(x_0, \rho)}$ is a compact subset of $\Omega$, and that $\deg (f, f(x_0), U_f(x_0, \rho)) \neq 0$. Then $f U_f(x_0, \rho) = \B^n(f(x_0), \rho)$.
	\end{lemma}
	\begin{proof}
		Since $U_f(x_0, \rho)$ is a component of $f^{-1} \B^n(f(x_0), \rho)$, we clearly have $f U_f(x_0, \rho) \subset \B^n(f(x_0), \rho)$. On the other hand, Lemma \ref{lem:top_degree_props} \eqref{enum:deg_constant} along with our assumption $\deg (f, f(x_0), U_f(x_0, \rho)) \neq 0$ implies that $\deg (f, y, U_f(x_0, \rho)) \neq 0$ for every $y \in \B^n(f(x_0), \rho)$. Thus, it follows from Lemma \ref{lem:top_degree_props} \eqref{enum:deg_image} that $\B^n(x_0, \rho) \subset f U_f(x_0, \rho)$.
	\end{proof}
	
	It follows that if $f^{-1}\{f(x_0)\}$ is discrete and $i(x_0, f) \neq 0$, then for small enough $\rho$, the sets $U_f(x_0, \rho)$ are strong normal neighborhoods of $x_0$.
	
	\begin{lemma}\label{lem:Ufs_are_strong_normal}
		Let $\Omega \subset \R^n$ be open, let $x_0 \in \Omega$, and let $f \colon \Omega \to \R^n$ be continuous. Suppose that $f^{-1}\{f(x_0)\}$ is discrete and $i(x_0, f) \neq 0$. Then there exists $\rho_0 = \rho_0(x_0, f_0)$ such that for every $\rho \leq \rho_0$, $U_f(x_0, \rho)$ is a strong normal neighborhood of $x_0$ under $f$. Moreover, $\diam U_f(x_0, \rho) \to 0$ as $\rho \to 0$.
	\end{lemma}
	\begin{proof}
		Since $f^{-1}\{f(x_0)\}$ is discrete, we may use \cite[Lemmas 3.1 and 3.4]{Kangasniemi-Onninen_1ptReshetnyak} to select a $\rho_1$ such that $U_f(x_0, \rho_1)$ is compactly contained in $\Omega$ and $\overline{U_f(x_0, \rho_1)} \cap f^{-1}\{f(x_0)\} = \{x_0\}$. Thus, the same is true for $U_f(x_0, \rho)$ whenever $\rho \in (0, \rho_1]$. By Lemma \ref{lem:boundary_into_boundary}, we also have $f \partial U_f(x_0, \rho) \subset \partial \B^n(f(x_0), \rho)$ for all $\rho > 0$. Moreover, by Lemma \ref{lem:Uf_surj}, the assumption $i(x_0, f) \neq 0$, and \eqref{eq:normal_neighborhood_degree}, we also have $f(x_0) \in \B^n(f(x_0), \rho) \subset \intr f U_f(x_0, \rho)$ whenever $\rho \in (0, \rho_1]$. Thus, $U_f(x_0, \rho)$ is a normal neighborhood of $x_0$ for all $\rho \leq \rho_1$.
		
		We then note that $\bigcap_{\rho > 0} \overline{U_f(x_0, \rho)}$ contains only points in $U_f(x_0, \rho_1)$ that $f$ maps to $f(x_0)$. The only such point is $x_0$, so we must have $\bigcap_{\rho > 0} \overline{U_f(x_0, \rho)} = \{x_0\}$. We thus obtain the claimed $\diam U_f(x_0, \rho) \to 0$ as $\rho \to 0$. Consequently, we may select $\rho_0 > 0$ such that there exists a ball $\B^n(x_0, r_0)$ with
		\[
			U_f(x_0, \rho_0) \subset \B^n(x_0, r_0) \subset U_f(x_0, \rho_1).
		\]
		Now, if $\rho \leq \rho_1$, $\R^n \setminus U_f(x_0, \rho)$ has only a single unbounded component which contains $\R^n \setminus \overline{\B^n}(x_0, r_0)$. So if $V$ is a bounded component of $\R^n \setminus U_f(x_0, \rho)$, then $V \subset \B^n(x_0, r_0) \subset U_f(x_0, \rho_1)$. Thus, for every $\rho \leq \rho_0$, $U_f(x_0, \rho)$ is a strong normal neighborhood of $x_0$.
	\end{proof}
	
	\subsection{Linear dilatations}\label{subsect:linear_dilatations}
	
	We recall the definition of linear dilatations for a continuous $f \colon \R^n \to \Omega$ at $x_0 \in \Omega$ from the introduction \eqref{eq:L_and_l}. That is, given $r > 0$, we let
	\begin{align*}
		L_f(x_0, r) &= \sup \left\{ \abs{y - f(x_0)} : y \in f \B^n(x_0, r)\right\} ,\\
		l_f(x_0, r) &= \inf \left\{ \abs{y - f(x_0)} : y \notin f \B^n(x_0, r)\right\}.
	\end{align*}
	In particular, $l_f(x_0, r)$ and $L_f(x_0, r)$ are respectively the largest and smallest radius such that
	\begin{equation}\label{eq:linear_distortion_ball_sandwich}
		\B^n(f(x_0), l_f(x_0, r)) \subset f(\B^n(x_0, r)) \subset \overline{\B^n}(f(x_0), L_f(x_0, r)).
	\end{equation}
	We also define the inverse linear dilatations of $f$ at $x_0$, which are instead defined for every $\rho > 0$ by
	\begin{align*}
		L_f^*(x_0, \rho) &= \sup \left\{ \abs{x - x_0} : x \in U_f(x_0, \rho) \right\} ,\\
		l_f^*(x_0, \rho) &= \inf \left\{ \abs{x - x_0} : x \notin U_f(x_0, \rho)\right\}.
	\end{align*}
	That is, analogously to \eqref{eq:linear_distortion_ball_sandwich}, $l_f^*(x_0, \rho)$ and $L_f^*(x_0, \rho)$ are respectively the largest and smallest radius such that
	\begin{equation}\label{eq:inverse_distortion_ball_sandwich}
		\B^n(x_0, l_f^*(x_0, \rho)) \subset U_f(x_0, \rho) \subset \B^n(x_0, L_f^*(x_0, \rho)).
	\end{equation}
	Note that a standard result yields that if $f \colon \Omega \to \R^n$ is quasiregular and non-constant, then
	\begin{align*}
		\limsup_{r \to 0} \frac{L_f(x_0, r)}{l_f(x_0, r)} \leq C < \infty
		\qquad \text{and} \qquad
		\limsup_{\rho \to 0} \frac{L_f^*(x_0, \rho)}{l_f^*(x_0, \rho)} \leq C < \infty,
	\end{align*}
	where $C = C(n, K, i(x_0, f))$; for details, see e.g.\ \cite[Section II.4]{Rickman_book}.
	
	We then point out the following result that ties the linear and inverse dilatations to each other.
	
	\begin{lemma}\label{lem:lin_inv_dil_cancel}
		Let $\Omega \subset \R^n$ be open, let $x_0 \in \Omega$, and let $f \colon \Omega \to \R^n$ be continuous. Let $\rho > 0$, and suppose that $\overline{U_f(x_0, \rho)}$ is a compact subset of $\Omega$. Then $L_f(x_0, l_f^*(x_0, \rho)) = \rho$.
	\end{lemma}
	\begin{proof}
		By \eqref{eq:inverse_distortion_ball_sandwich}, $\B^n(x_0, l_f^*(x_0, \rho)) \subset U_f(x_0, \rho)$. Hence, $f\B^n(x_0, l_f^*(x_0, \rho)) \subset \B^n(f(x_0), \rho)$, and consequently $L_f(x_0, l_f^*(x_0, \rho)) \leq \rho$. 
		
		On the other hand, since $\B^n(x_0, l_f^*(x_0, \rho))$ is the largest ball contained in $U_f(x_0, \rho)$, we must have that $\partial \B^n(x_0, l_f^*(x_0, \rho))$ meets $\partial U_f(x_0, \rho)$, which in turn is mapped by $f$ into $\partial \B^n(f(x_0), \rho)$. Thus, $\abs{f(x_1) - f(x_0)} = \rho$ for some $x_1 \in \partial \B^n(x_0, l_f^*(x_0, \rho))$. As $f$ is continuous in $\overline{\B^n}(x_0, l_f^*(x_0, \rho))$, we thus conclude that $L_f(x_0, l_f^*(x_0, \rho)) \geq \rho$.
	\end{proof}
	
	Note that for our definitions of linear and inverse dilatations, this is the only equality of this type without extra assumptions on $f$ like openness: in general one only has $l_f(x_0, L_f^*(x_0, \rho)) \geq \rho$, and $r$ can be either larger or smaller than $L^*_f(x_0, l_f(x_0, r))$ or $l^*_f(x_0, L_f(x_0, r))$.
	
	\subsection{Intersections with boundaries of balls}
	
	We also use several technical lemmas on the interaction of $\B^n(x_0, r)$ and $U_f(x_0, \rho)$ if $l_f(x_0, r) < \rho <  L_f(x_0, r)$. The essential result is contained in the following lemma. 
	
	\begin{lemma}\label{lem:inner_and_outer_radius_prelemma}
		Let $\Omega \subset \R^n$ be open, let $x_0 \in \Omega$, and let $f \colon \Omega \to \R^n$ be continuous. Suppose that $f^{-1} \{f(x_0)\}$ is discrete, and that $i(x_0, f) \neq 0$. Let $r > 0$, and let $\rho > 0$ be such that $U_f(x_0, \rho)$ is a normal neighborhood of $x_0$. If $\rho > l_f(x_0, r)$, then $U_f(x_0, \rho) \setminus \B^n(x_0, r) \neq \emptyset$. Similarly, if $\rho < L_f(x_0, r)$, then $\B^n(x_0, r) \setminus U_f(x_0, \rho) \neq \emptyset$.
	\end{lemma}
	\begin{proof}
		Since $U_f(x_0, \rho)$ is a normal neighborhood of $x_0$ and $i(x_0, f) \neq 0$, we have $\B^n(f(x_0), \rho) = f U_f(x_0, \rho)$ by Lemma \ref{lem:Uf_surj}. We then observe that if $\rho > l_f(x_0, r)$, there must exist a point $y_1 \in \B^n(f(x_0), \rho) = fU_f(x_0, \rho)$ such that $y_1 \notin f\B^n(x_0, r)$. This is impossible if $U_f(x_0, \rho) \subset \B^n(x_0, r)$, so we conclude that $U_f(x_0, \rho) \setminus \B^n(x_0, r) \neq \emptyset$. Similarly, if $\rho < L_f(x_0, r)$, there must exist a point $y_2 \in f\B^n(x_0, r)$ such that $y_2 \notin \B^n(f(x_0), \rho) = f U_f(x_0, \rho)$. This is in turn impossible if $\B^n(x_0, r) \subset U_f(x_0, \rho)$, so we conclude that $\B^n(x_0, r) \setminus U_f(x_0, \rho) \neq \emptyset$.
	\end{proof}
	
	For the inner dilatation, the following corollary follows nearly immediately.
	
	\begin{cor}\label{cor:inner_radius_lemma}
		Let $\Omega \subset \R^n$ be open, let $x_0 \in \Omega$, and let $f \colon \Omega \to \R^n$ be continuous. Suppose that $f^{-1} \{f(x_0)\}$ is discrete, and that $i(x_0, f) \neq 0$. Let $r > 0$, and let $\rho > 0$ be such that $U_f(x_0, \rho)$ is a normal neighborhood of $x_0$. If $\rho > l_f(x_0, r)$, then $U_f(x_0, \rho) \cap \partial \B^n(x_0, r) \neq \emptyset$.
	\end{cor}
	\begin{proof}
		The set $U_f(x_0, \rho)$ is connected, contains $x_0$, and meets $\R^n \setminus \B^n(x_0, r)$ by Lemma \ref{lem:inner_and_outer_radius_prelemma}. Therefore, $U_f(x_0, \rho)$ must intersect $\partial\B^n(x_0, r)$.
	\end{proof}
	
	However, for the outer dilatation, the situation is complicated by the fact that $\R^n \setminus U_f(x_0, \rho)$ could be disconnected for a non-open $f$. A major part of the proof of Theorem \ref{thm:limsup_lin_distortion} is about finding a way to work around this disconnectedness of $\R^n \setminus U_f(x_0, \rho)$.
	
	\section{Preliminaries on quasiregular values}
	
	We then recall our recent results on quasiregular values that this work is built on. The first major result we recall is the single-value Reshetnyak's theorem for quasiregular values, which is the main theorem of \cite{Kangasniemi-Onninen_1ptReshetnyak}.
	
	\begin{thm}[{\cite[Theorem 1.2]{Kangasniemi-Onninen_1ptReshetnyak}}]\label{thm:1p_Reshetnyak}
		Let $\Omega \subset \R^n$ be a domain and let $f \in W^{1,n}_\loc(\Omega, \R^n)$ be continuous. Suppose that $f$ has a $(K, \Sigma)$-quasiregular value at $y_0 \in \R^n$, where $K \geq 1$ and $\Sigma \in L^{1+\eps}_\loc(\Omega)$ with $\eps > 0$. Then either $f \equiv y_0$, or the following three conditions hold:
		\begin{enumerate}
			\item \label{enum:Reshetnyak_discrete} (Discreteness at $y_0$) $f^{-1}\{y_0\}$ is a discrete set;
			\item \label{enum:Reshetnyak_index} (Sense-preserving at $y_0$) for every $x_0 \in f^{-1}\{y_0\}$, $i(f, x_0) > 0$;
			\item \label{enum:Reshetnyak_open} (Openness at $y_0$) for every $x_0 \in f^{-1}\{y_0\}$, if $U \subset \Omega$ is an open neighborhood of $x_0$, then $y_0 \in \intr f(U)$.
		\end{enumerate}
	\end{thm}
	
	We then also recall that mappings with quasiregular values inherit higher Sobolev regularity from $\Sigma$. The proof is a standard argument based on proving a reverse H\"older inequality and applying Gehring's lemma. The result we state here is \cite[Lemma 6.1]{Kangasniemi-Onninen_1ptReshetnyak}, though variants in slightly different settings can also be found in \cite[Lemma 2.1]{Dolezalova-Kangasniemi-Onninen_MGFD-cont} and \cite[Lemma 4.5]{Kangasniemi-Onninen_Picard}.
	
	\begin{lemma}\label{lem:qrval_higher_int}
		Let $\Omega \subset \R^n$ be a domain and let $f \in W^{1,n}_\loc(\Omega, \R^n)$. Suppose that $f$ has a $(K, \Sigma)$-quasiregular value at $y_0 \in \R^n$, where $K \geq 1$ and $\Sigma \in L^{1+\eps}_\loc(\Omega)$ with $\eps > 0$. Then $f \in W^{1,p}_\loc(\Omega, \R^n)$ for some $p > n$.
	\end{lemma}
	
	We note that by e.g.\ \cite[Lemma 8.1]{Bojarski-Iwaniec_QR}, we obtain the following consequence.
	
	\begin{lemma}\label{lem:qrval_Lusin_N}
		Let $\Omega \subset \R^n$ be a domain and let $f \in W^{1,n}_\loc(\Omega, \R^n)$. Suppose that $f$ has a $(K, \Sigma)$-quasiregular value at $y_0 \in \R^n$, where $K \geq 1$ and $\Sigma \in L^{1+\eps}_\loc(\Omega)$ with $\eps > 0$. Then $f$ satisfies the Lusin (N) -condition.
	\end{lemma}
	
	\subsection{Uniform H\"older continuity}
	
	It follows from Lemma \ref{lem:qrval_higher_int} that if a map $f \in W^{1,n}_\loc(\Omega, \R^n)$ has a $(K, \Sigma)$-quasiregular value with $\Sigma \in L^{1+\eps}_\loc(\Omega)$, $\eps > 0$, then the map has a H\"older continuous representative. However, this approach does not yield the optimal degree of H\"older continuity of such mappings. For this, one instead uses an alternate approach presented in \cite[Theorem 1.1]{Kangasniemi-Onninen_Heterogeneous}; see also \cite[Theorem 1.2]{Dolezalova-Kangasniemi-Onninen_MGFD-cont} for a version under even weaker regularity assumptions for $\Sigma$.
	
	In the proof of the rescaling principle, our arguments use a uniform version of these H\"older continuity results. This result has not been explicitly stated previously, but it follows immediately from the proof of \cite[Theorem 1.1]{Kangasniemi-Onninen_Heterogeneous} if one tracks the exact dependencies of the constants. We state a version that is sufficient for us and outline the main points of the proof.
	
	\begin{lemma}\label{lem:qrval_local_holder}
		Let $f \in W^{1,n}(\B^n, \R^n) \cap L^\infty(\B^n, \R^n)$ be continuous. Suppose that $f$ has a $(K, \Sigma)$-quasiregular value at $y_0 \in \R^n$, where $K \geq 1$ and $\Sigma \in L^{1+\eps}(\B^n)$ with 
		\[
			0 < \eps < \frac{1}{K-1},
		\]
		where we interpret $1/(K-1) = \infty$ if $K = 1$. Let $A \geq 0$ be such that
		\[
			\max \left( \norm{f-y_0}_{L^\infty(\B^n)}, \norm{Df}_{L^n(\B^n)}, \norm{\Sigma}_{L^{1+\eps}(\B^n)} \right) \leq A
		\] 
		Then for every $r \in (0, 1)$, we have
		\[
			\abs{f(x) - f(y)} \leq C(n, A, \eps, r) \abs{x-y}^\frac{\eps}{1+\eps}
		\] 
		for all $x, y \in \B^n(0, r)$.
	\end{lemma}
	\begin{proof}
		Let $r \in (0, 1)$, and let $s_0 = (1-r)/2$. We fix an $x_1 \in \B^n(0, r)$, and note that $\B^n(x_1, s_0)$ is compactly contained in $\B^n$. For $s \leq s_0$, we denote $B_{s} = \B^n(x_1, s)$ and
		\[
			\Phi(s) = \int_{B_s} \abs{Df}^n.
		\]
		Note that $\Phi$ is defined on $[0, s_0]$, and 
		\begin{equation}\label{eq:Phi_maximum_bound}
			\Phi(s_0) \leq \int_{\B^n} \abs{Df}^n \leq A.
		\end{equation}
		
		It is shown in \cite[Lemma 3.1]{Kangasniemi-Onninen_Heterogeneous} using the isoperimetric inequality of Sobolev functions and H\"older's inequality that
		\[
		\Phi(s) \leq \frac{Ks}{n} \Phi'(s) + \int_{B_s} \abs{f - y_0}^n \Sigma_r.
		\]
		for a.e.\ $s \in (0, s_0)$. The last term can be estimated by
		\[
			\int_{B_s} \abs{f - y_0}^n \Sigma_r 
			\leq \norm{f-y_0}_{L^\infty(\B^n)} \norm{\Sigma}_{L^{1+\eps}(\B^n)} \abs{B_s}^\frac{\eps}{1 + \eps}
			\leq C(n, \eps) A^2 s^\frac{n\eps}{1+\eps}.
		\]
		Thus, we have an estimate on $\Phi_r(s)$ of the form
		\begin{equation}\label{eq:Phi_estimate}
			\Phi(s) \leq \frac{Ks}{n} \Phi'(s) +  \leq C(n, \eps) A^2 s^\frac{n\eps}{1+\eps}.
		\end{equation}
		
		Now, since $\Phi \colon [0, s_0] \to [0, A]$ satisfies \eqref{eq:Phi_estimate}, and since $n\eps/(1+\eps) < n/K$ by our assumption that $\eps < 1/(K-1)$, we may apply \cite[Lemma 3.3]{Kangasniemi-Onninen_Heterogeneous} to obtain the estimate
		\begin{equation}\label{eq:Phi_estimate_2}
			\Phi(s) \leq C(n, \eps, s_0, A) s^\frac{n\eps}{1+\eps}.
		\end{equation}
		Since \eqref{eq:Phi_estimate_2} holds for all $x_1 \in \B^n(0, r)$ and $s \in [0, s_0]$, we may thus use a modulus of continuity estimate by Morrey \cite[Theorem 3.5.2]{Morrey-Book} to conclude that
		\[
		\abs{h_r(z) - h_r(z')} \leq C(n, \eps, s_0, A) \abs{z-z'}^\frac{\eps}{1+\eps},
		\]
		for all $x_1 \in \B^n(0, r)$ and $z, z' \in \B^n(x_1, s_0/3)$. Since $s_0 = (1-r)/2$ is dependent only on $r$, we obtain the desired H\"older continuity estimate on $\B^n(x_0, r)$ by covering $\B^n(x_0, r)$ with finitely many balls of the form $\B^n(x_1, s_0/3)$ and by chaining together finitely many of the above H\"older estimates.
	\end{proof}
	
	\subsection{Level set integrals}
	
	Results in the theory of quasiregular values often rely on using various arguments based on degree theory to analyze integrals of $J_f/\abs{f - y_0}^n$. We collect the key results of this type that we need into the following Lemma. The proof is a relatively standard argument based on degree theory and changes of variables, but we regardless present the argument
	
	\begin{lemma}\label{lem:level_set_lemma}
		Let $\Omega \subset \R^n$ be a domain, let $x_0 \in \Omega$, and let $f \in W^{1,n}_\loc(\Omega, \R^n)$ be a non-constant continuous map. Suppose that $f$ has a $(K, \Sigma)$-quasiregular value at $f(x_0)$, where $K \geq 1$ and $\Sigma \in L^{1+\eps}_\loc(\Omega)$ with $\eps > 0$. Suppose that $\rho$ is small enough that $U_f(x_0, \rho)$ is a strong normal neighborhood of $x_0$ under $f$. Then the following results hold.
		\begin{enumerate}[label=(\roman*)]
			\item\label{enum:level_set_integral_over_V} if $V_\rho$ is a bounded component of $\R^n \setminus U_f(x_0, \rho)$, then
			\[
			\int_{V_\rho} \frac{J_f}{\abs{f - f(x_0)}^n} = 0.
			\] 
			\item\label{enum:level_set_integral_over_U} If $\rho' < \rho$, then
			\[
			\int_{U_f(x_0, \rho) \setminus U_f(x_0, \rho')} \frac{J_f}{\abs{f - f(x_0)}^n} = C(n) i(x_0, f) \log \frac{\rho}{\rho'}.
			\] 
		\end{enumerate}
	\end{lemma}
	\begin{proof}
		By Theorem \ref{thm:1p_Reshetnyak}, $f^{-1}\{f(x_0)\}$ is a discrete subset of $\Omega$ and $i(f, x) > 0$ for every $x \in f^{-1}\{f(x_0)\}$. Thus, by Lemma \ref{lem:Ufs_are_strong_normal}, $U_f(x_0, \rho)$ is indeed a strong normal neighborhood of $x_0$ for small enough $\rho$. 
		
		We then fix $\rho$, and let $V_\rho$ be a bounded component of $\R^n \setminus U_f(x_0, \rho)$. Observe that $\partial V_\rho \subset \partial U_f(x_0, \rho)$, and hence $f \partial V_\rho \subset \partial \B^n(f(x_0), \rho)$. By Lemma \ref{lem:normal_nbhd_props}, $f(x_0) \notin fV_\rho$. Thus, $\deg(f, y, \intr V_\rho) = \deg(f, f(x_0), \intr V_\rho) = 0$ for all $y \in \B^n(f(x_0), \rho)$ by parts \eqref{enum:deg_constant} and \eqref{enum:deg_image} of Lemma \ref{lem:top_degree_props}. Moreover, since $V_\rho$ is a compact subset of $\Omega$, $\abs{f - f(x_0)}$ has a maximal value $M$ on it, and if we pick any $y_1 \notin \B^n(f(x_0), M+1)$, we can similarly argue that $\deg(f, y, \intr V_\rho) = \deg(f, y_1, \intr V_\rho) = 0$ for all $y \notin \overline{\B^n}(f(x_0), \rho)$. 
		
		Now, we apply the change of variables -formula of Lemma \ref{lem:change_of_vars}, wherein we use $v(y) = \min(\abs{y - f(x_0)}^{-n}, m^{-n})$ with $m = \min_{x \in V_\rho} \abs{f(x) - f(x_0)} > 0$ to ensure $v \in L^\infty(\R^n)$. Thus,
		\[
		\int_{\intr V_\rho} \frac{J_f}{\abs{f - f(x_0)}^n}
		= \int_{\R^n \setminus \partial \B^n(f(x_0), \rho)} v \cdot \deg(f, \cdot, \intr(V_\rho)) = 0.
		\]
		It thus remains to consider the integral of $J_f/\abs{f - f(x_0)}^n$ over $\partial V_\rho$. However, the image $f \partial V_\rho$ of $V_\rho$ is contained in the null-set $\partial \B^n(f(x_0), \rho)$, so by e.g.\ \cite[Theorems 5.6 and 5.21]{Fonseca-Gangbo-book}, $J_f$ vanishes a.e.\ in $\partial V_\rho$. Thus,
		\[
		\int_{\partial V_\rho} \frac{J_f}{\abs{f - f(x_0)}^n} = 0,
		\]
		completing the proof of \ref{enum:level_set_integral_over_V}.
		
		It remains to prove \ref{enum:level_set_integral_over_U}. Suppose that $\rho' < \rho$, and denote $D = U_f(x_0, \rho) \setminus \overline{U_f(x_0, \rho')}$. Then $\deg(f, y, D) = 0$ for $y \notin \overline{\B^n}(f(x_0), \rho)$ by the image set part \eqref{enum:deg_image} of Lemma \ref{lem:top_degree_props}. Similarly, if $y \in \B^n(f(x_0), \rho) \setminus \overline{\B^n}(f(x_0), \rho')$, then $\deg(f, y, D) = i(x_0, f)$ by \eqref{eq:normal_neighborhood_degree} along with the excision and local constancy parts \eqref{enum:deg_excision} and \eqref{enum:deg_constant} of Lemma \ref{lem:top_degree_props}. Moreover, since $f^{-1}\{f(x_0)\} \cap D = \emptyset$, $\deg(f, y, D) = 0$ for $y \in \B^n(f(x_0), \rho')$ by the local constancy and image set parts \eqref{enum:deg_constant} and \eqref{enum:deg_image} of Lemma \ref{lem:top_degree_props}.
		
		By another application of the change of variables formula from Lemma \ref{lem:change_of_vars} as before, we obtain
		\begin{multline*}
			\int_D \frac{J_f}{\abs{f - f(x_0)}^n} = \int_{\B^n(f(x_0), \rho) \setminus \B^n(f(x_0), \rho')} \frac{i(x_0, f)}{\abs{y - f(x_0)}^n} \vol_n(y)\\
			= C(n) i(x_0, f) \int_{\rho'}^\rho \frac{r^{n-1} \dd r}{r^n}
			= C(n) i(x_0, f) \log \frac{\rho}{\rho'}.
		\end{multline*}
		Moreover, the only difference between $D$ and $U_f(x_0, \rho) \setminus U_f(x_0, \rho')$ is the exclusion of $\partial U_f(x_0, \rho')$, but $J_f$ again vanishes a.e.\ in $\partial U_f(x_0, \rho')$ since $f(\partial U_f(x_0, \rho'))$ is contained in the nullset $\partial \B^n(f(x_0), \rho')$. Thus, the proof of \ref{enum:level_set_integral_over_U} is complete.
	\end{proof}
	
	\subsection{A higher regularity result}
	
	We also use the following higher regularity result.
	
	\begin{lemma}\label{lem:reverse_Holder}
		Let $\Omega \subset \R^n$ be open, let $y_0 \in \R^n$, and let $f \in W^{1,n}_\loc(\Omega, \R^n)$ be a non-constant continuous map with $y_0 \notin f \Omega$. Suppose that $f$ has a $(K, \Sigma)$-quasiregular value at $y_0$, where $K \geq 1$ and $\Sigma \in L^{1+\eps}_\loc(\Omega)$ with $\eps > 0$. Then for every cube $Q$ such that $2Q \subset \Omega$, we have
		\[
			\dashint_{Q} \frac{\abs{Df}^n}{\abs{f-y_0}^n} \lesssim_n K \left( \dashint_{2Q} \frac{\abs{Df}^{\frac{n^2}{n + 1}}}{\abs{f-y_0}^{\frac{n^2}{n + 1}}} \right)^\frac{n + 1}{n} + \dashint_{2Q} \Sigma.
		\]
	\end{lemma}
	\begin{proof}[Idea of proof]
		The claim essentially follows from the proof of \cite[Lemma~4.5]{Kangasniemi-Onninen_Picard} with minimal changes.
		
		Indeed, since $y_0 \notin f\Omega$, we can define a continuous function $G \colon U \to \R \times \S^{n-1}$ by
		\[
			G(x) = \left( \log \abs{f(x) - y_0}, \frac{f(x) - y_0}{\abs{f(x) - y_0}} \right);
		\]
		for details on this ``spherical logarithm'' of $f - y_0$, see \cite[Section 7]{Kangasniemi-Onninen_Heterogeneous} along with the correction \cite{Kangasniemi-Onninen_Heterogeneous-corrigendum}, as well as \cite[Section 4]{Kangasniemi-Onninen_Picard}. The key property of the map $G$ is that it lies in $W^{1,n}_\loc(\Omega, \R \times \S^{n-1})$ with
		\[
			\abs{DG} = \frac{\abs{Df}}{\abs{f - y_0}} \qquad \text{and} \qquad J_G = \frac{J_f}{\abs{f - y_0}^n},
		\]
		and therefore the fact that $f$ has a $(K, \Sigma)$-quasiregular value at $y_0$ implies that
		\begin{equation}\label{eq:G_heterogeneous_estimate}
			\abs{DG}^n \leq K J_G + \Sigma
		\end{equation}
		a.e.\ in $U$. 
		
		The claim is hence a reverse H\"older inequality for $\abs{DG}$, and can be proven by using \eqref{eq:G_heterogeneous_estimate} and a Caccioppoli-type inequality for functions into $\R \times \S^{n-1}$ shown e.g.\ in \cite[Lemma 2.3]{Kangasniemi-Onninen_Heterogeneous}. The proof is given in \cite[Lemma~4.5]{Kangasniemi-Onninen_Picard} in the case $G \in W^{1,n}_\loc(\R^n, \R \times \S^{n-1})$. However, the same proof also works on a smaller domain $\Omega$ as long as we assume $2Q \subset \Omega$.
	\end{proof}
	
	Thus, we obtain the following corollary by combining Lemma \ref{lem:reverse_Holder} with the version of Gehring's lemma recalled in Proposition \ref{prop:local_Gehring_lemma}.
	
	\begin{cor}\label{cor:higher_int}
		Let $\Omega \subset \R^n$ be open, let $y_0 \in \R^n$, and let $f \in W^{1,n}_\loc(\Omega, \R^n)$ be a non-constant continuous map with $y_0 \notin f \Omega$. Suppose that $f$ has a $(K, \Sigma)$-quasiregular value at $y_0$, where $K \geq 1$ and $\Sigma \in L^{1+\eps}(\Omega)$ with $\eps > 0$. Suppose also that $\abs{Df}/\abs{f-y_0} \in L^n(\Omega)$. Then there exists $p = p(n, K) \in (n, n(1+\eps))$ such that for every cube $Q$ with $\sqrt{2} Q \subset \Omega$, we have
		\[
			\int_{Q} \frac{\abs{Df}^p}{\abs{f-f(x_0)}^p}  \leq C(n, p) \left( \abs{Q}^{-\frac{p-n}{n}} \left( \int_{\sqrt{2} Q} \frac{\abs{Df}^n}{\abs{f-f(x_0)}^n} \right)^\frac{p}{n} + \int_{\sqrt{2} Q} \Sigma^{\frac{p}{n}}  \right).
		\]
	\end{cor}
	
	\section{Bounded components of normal neighborhoods}\label{sect:bounded_components}
	
	If $\Omega \subset \R^n$ is open and $f \colon \Omega \to \R^n$ is a continuous and open map, then it is possible to show that for small enough $\rho$, the set $\R^n \setminus U_f(x_0, \rho)$ is a connected unbounded set. However, if $f \colon \Omega \to \R^n$ is merely a continuous map, then it is entirely possible that $\R^n \setminus U_f(x_0, \rho)$ contains multiple components, some of them unbounded.
	
	Our objective in this section is to prove the following key lemma that, in the presence of a $(K, \Sigma)$-quasiregular value at $f(x_0)$, grants us control over how far $f$ can extend on a bounded component of $\R^n \setminus U_f(x_0, \rho)$. 
	
	\begin{lemma}\label{lem:the_important_distortion_lemma}
			Let $\Omega \subset \R^n$ be a domain, let $x_0 \in \Omega$, let $f \in W^{1,n}_\loc(\Omega, \R^n)$ be a non-constant continuous map, and let $\alpha > 1$. Suppose that $f$ has a $(K, \Sigma)$-quasiregular value at $f(x_0)$, where $K \geq 1$ and $\Sigma \in L^{1+\eps}_\loc(\Omega)$ with $\eps > 0$. Then there exists a $\rho_0 = \rho_0(\Omega, f, x_0, \alpha, K, \Sigma)$ such that, if $\rho \in (0, \rho_0)$ and $V$ is a bounded component of $\R^n \setminus U_f(x_0, \rho)$, we have
			\[
				\sup_{x \in V} \abs{f(x) - f(x_0)} \leq \alpha^2 \rho.
			\]
	\end{lemma}

	The proof is somewhat involved, and a significant portion of our more specific preliminary results, such as Lemmas \ref{lem:basic_adm_function_estimate} and \ref{lem:level_set_lemma} along with Corollary \ref{cor:higher_int}, have been formulated specifically for the proof of this result. The proof proceeds in two steps. In the first step, we show that if $\alpha > 1$ and $V$ is a bounded component of $\R^n \setminus U_f(x_0, \rho)$ with $\rho$ small enough, then every bounded component $V'$ of $\R^n \setminus U_f(x_0, \alpha \rho)$ that is contained in $V$ is in fact contained in a sufficiently small ball within $V$. In the second step, we then further limit how much further $f$ can escape on $V'$ with an argument based on logarithmic higher integrability and the fact that the $p$-capacity of a point is positive when $p > n$.
	
	\subsection{Step 1: Roundness}
	
	We collect the statement we wish to prove in the first step in the following lemma; see also the illustration in Figure \ref{fig:roundness}.
	
	\begin{lemma}\label{lem:roundness_lemma}
		Let $\Omega \subset \R^n$ be a domain, let $x_0 \in \Omega$, let $f \in W^{1,n}_\loc(\Omega, \R^n)$ be a non-constant continuous map, and let $\alpha > 1$. Suppose that $f$ has a $(K, \Sigma)$-quasiregular value at $f(x_0)$, where $K \geq 1$ and $\Sigma \in L^{1+\eps}_\loc(\Omega)$ with $\eps > 0$. Then there exists $\rho_0 = \rho_0(\Omega, f, x_0, \alpha, \Sigma)$ as follows: if $\rho \in (0, \rho_0)$, $V$ is a bounded component of $\R^n \setminus U_f(x_0, \rho)$, and $V' \subset V$ is a non-empty bounded component of $\R^n \setminus U_f(x_0, \alpha \rho)$, then there exist a point $x_1 \in V$ and a radius $r_1 > 0$ such that
		\[
			V' \subset \overline{\B^n}(x_1, r_1) \qquad \text{and} \qquad \B^n(x_1, 2r_1) \subset V.
		\]
	\end{lemma}
	
	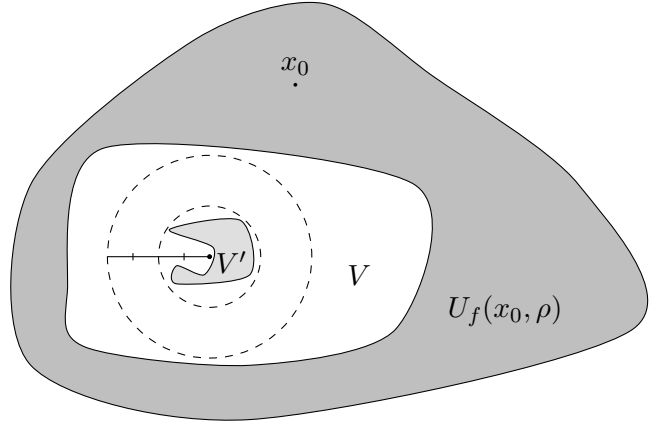
\begin{figure}[h]
		\begin{tikzpicture}[scale = 2.4]
			\draw [fill=gray!50] plot [smooth cycle,tension=0.5] coordinates 
			{(-0.2, 0) (1, -0.3) (3.1, 0.2) (2.8, 1) (2, 1.6)
				(1.4, 2) (0.7, 1.8) (-0.2, 1)};
			\draw [fill=white] plot [smooth cycle,tension=0.5] coordinates 
			{(0.1, 0.1) (1, 0) (1.8, 0.2) 
				(1.9, 1.0) (0.2, 1.2) (0, 0.4)};
			\draw [fill=lightgray!50] plot [smooth cycle,tension=0.5] coordinates 
			{(0.6, 0.45) (1, 0.5) (0.95, 0.8) (0.56, 0.75) 
				(0.8, 0.65) (0.75, 0.5) (0.6, 0.55)};
				
			\draw[dashed] (0.78, 0.6) circle (0.28);
			\draw[dashed] (0.78, 0.6) circle (0.56);
			\filldraw (0.78, 0.6) circle (0.3pt);
			\draw (0.78, 0.6) -- (0.78 - 0.56, 0.6);
			
			\draw (0.78 - 0.56*0.25, 0.6+0.02) -- 
				(0.78 - 0.56*0.25, 0.6-0.02);
			\draw (0.78 - 0.56*0.75, 0.6+0.02) -- 
				(0.78 - 0.56*0.75, 0.6-0.02);

			\fill (1.25,1.55) circle[radius=0.3pt] node[anchor=south] {$x_0$};
			\draw (2.4, 0.3) node[anchor=center] {$U_f(x_0, \rho)$};
			\draw (0.9, 0.57) node[anchor=center] {$V'$};
			\draw (1.6, 0.5) node[anchor=center] {$V$};
			
		\end{tikzpicture}
		\caption{Illustration of the result of Lemma \ref{lem:roundness_lemma}: for small enough $\rho$, the component $V'$ of $\R^n \setminus U_f(x_0, \alpha\rho)$ is contained inside the component $V$ of $\R^n \setminus U_f(x_0, \rho)$ in a relatively controlled manner.}\label{fig:roundness}
	\end{figure}
	
	\begin{proof}
		By the single-point Reshetnyak's theorem recalled in Theorem \ref{thm:1p_Reshetnyak}, we have that $f^{-1}\{f(x_0)\}$ is discrete and $i(x_0, f) > 0$. We choose $\rho_0$ to be small enough that if $\rho \in (0, \rho_0)$, then $U_f(x_0, \alpha\rho)$ is a strong normal neighborhood of $x_0$. We also require that $\B^n(x_0, \diam U_f(x_0, \alpha \rho_0)) \subset \Omega$ and
		\begin{equation}\label{eq:roundness_lemma_setup}
			\int_{\B^n(x_0, \diam U_f(x_0, \alpha \rho_0))} \Sigma \leq \frac{C(n) \log^n(\alpha) \log(5/4)}{2^n }, 
		\end{equation}
		where $C(n)$ is the constant from the version of the Loewner property given in Lemma \ref{lem:Loewner}. These choices are possible by Lemma \ref{lem:Ufs_are_strong_normal} along with the $L^1_\loc$-integrability of $\Sigma$.
		
		Suppose then that $\rho \in (0, \rho_0)$, and let $V$, $V'$ be bounded non-empty components of $\R^n \setminus U_f(x_0, \rho)$ and $\R^n \setminus U_f(x_0, \alpha \rho)$, respectively, with $V' \subset V$. We denote $U_0 = U_f(x_0, \rho)$, $V_0 = V$, $U_2 = U_f(x_0, \alpha \rho)$ and $V_2 = V'$. We also select an intermediary normal neighborhood $U_1 = U_f(x_0, \sqrt{\alpha} \rho)$, and let $V_1$ denote the component of $\R^n \setminus U_1$ for which $V_2 \subset V_1 \subset V_0$. We also denote $D_i = \diam(V_i)$ for $i = 0, 1, 2$, and $d_i = \dist(V_{i}, U_{i-1})$ for $i = 1, 2$.  See Figure \ref{fig:Ui_and_Vi} for an illustration. Note that we may assume that $D_2 > 0$, as otherwise $V' = V_2$ is a point $\{x_1\}$, and the claim follows by selecting any $r_1 > 0$ such that $\B^n(x_1, 2r_1) \subset \intr V$.
		
		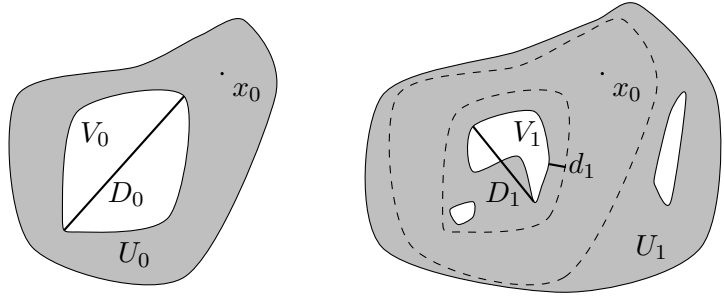
\begin{figure}[h]
			\begin{tikzpicture}[scale = 2]
				\begin{scope}
					\draw [fill=gray!50] plot [smooth cycle,tension=0.5] coordinates 
						{(0, 0) (1, -0.2) (1.6, 1) (1.4, 1.5)
						(1.1, 1.4) (0.7, 1.2) (-0.1, 1)};
					\draw [fill=white] plot [smooth cycle,tension=0.5] coordinates 
						{(0.3, 0.1) (0.9, 0.2) (1.0, 1.0) (0.3, 0.9) (0.2, 0.2)};
					\fill (1.25,1.15) circle[radius=0.3pt] node[anchor=north west] {$x_0$};
				
					\draw (0.5, -0.05) node[anchor=west] {$U_0$};
					\draw (0.25, 0.75) node[anchor=west] {$V_0$};
				
					\draw[thick] (0.21, 0.11) -- (1.0, 1.0);
					\draw (0.62, 0.5) node[anchor=north] {$D_0$};
				\end{scope}
				
				\begin{scope}[shift={(2.5,0)}]
					\draw [fill=gray!50] plot [smooth cycle,tension=0.5] coordinates 
					{(-0.2, 0) (1, -0.3) (1.9, 0) (2, 1) (1.5, 1.6)
						(1.2, 1.5) (0.7, 1.3) (-0.2, 1)};
					\draw [fill=white] plot [smooth cycle,tension=0.5] coordinates 
					{(1.7, 0.3) (1.8, 1) (1.7, 0.9) (1.6, 0.5) (1.6, 0.4)};
					\draw [fill=white] plot [smooth cycle,tension=0.5] coordinates 
					{(0.3, 0.15) (0.4, 0.2) (0.4, 0.3) (0.25, 0.25)};
					\draw [fill=white] plot [smooth cycle,tension=0.5] coordinates 
					{(0.4, 0.5) (0.7, 0.6) (0.8, 0.3) (0.85, 0.4) (0.9, 0.6)
						(0.8, 0.9) (0.4, 0.8)};
					
					\draw [dashed] plot [smooth cycle,tension=0.5] coordinates 
					{(0, 0) (1, -0.2) (1.6, 1) (1.4, 1.5)
						(1.1, 1.4) (0.7, 1.2) (-0.1, 1)};
					\draw [dashed] plot [smooth cycle,tension=0.5] coordinates 
					{(0.3, 0.1) (0.9, 0.2) (1.0, 1.0) (0.3, 0.9) (0.2, 0.2)};
					\fill (1.25,1.15) circle[radius=0.3pt] node[anchor=north west] {$x_0$};
					
					\draw (1.4, 0) node[anchor=west] {$U_1$};
					\draw (0.76, 0.9) node[anchor=north] {$V_1$};
					
					\draw[thick] (0.9, 0.55) -- (1.01, 0.53);
					\draw (0.96, 0.54) node[anchor=west] {$d_1$};
					\draw[thick] (0.4, 0.8) -- (0.8, 0.3);
					\draw (0.6, 0.5) node[anchor=north] {$D_1$};
				\end{scope}
			\end{tikzpicture}
			\caption{An example illustration of the sets $U_0$, $V_0$, $U_1$ and $V_1$, along with the quantities $D_0 = \diam(V_0)$, $D_1 = \diam(V_1)$, and $d_1 = \dist(V_1, U_0)$.}\label{fig:Ui_and_Vi}
		\end{figure}
		
		We first show that our claim holds if $2D_2 < d_1 + d_2$. Indeed, we have $\dist(V_2, U_0) \geq \dist(V_1, U_0) + \dist(V_2, U_1) = d_1 + d_2$, since any path from $U_0$ to $V_2$ has to first intersect the boundary of $V_1$ before reaching $V_2$. Thus, if we select any $x_1 \in V_2$, we get $V_2 \subset \overline{\B^n}(x_1, D_2)$ and $\B^n(x_1, 2D_2) \subset \B^n(V_2, 2D_2) \subset \B^n(V_2, d_1 + d_2) \subset \B^n(V_2, \dist(V_2, U_0)) \subset V_0$, proving that the claim is true with $r_1 = D_2$.
		
		Thus, we assume that
		\begin{equation}\label{eq:chain_part_1}
			2D_2 \geq d_1 + d_2,
		\end{equation}
		with the objective of deriving a contradiction. In addition to this estimate, we observe that 
		\begin{equation}\label{eq:chain_part_2}
			D_{i} \geq 2 d_{i+1}
		\end{equation}
		for all $i \in \{0, 1\}$. Indeed, if $x \in V_{i+1}$, then $\B^n(x, d_{i+1}) \subset V_{i}$, and $\B^n(x, d_{i+1})$ has diameter $2d_{i+1}$.
		
		We then observe that for $i \in \{1, 2\}$, the sets $V_i$ and $\overline{U_{i-1}}$ are connected, and their distance is $d_i$. Moreover, since $V_i$ is contained in a bounded component $V_{i-1}$ of $\R^n \setminus U_{i-1}$, we must have $D_i = \diam(V_i) < \diam \overline{U_{i-1}}$. Thus, we may apply the Loewner property from Lemma \ref{lem:Loewner} as well as the monotonicity of capacity stated in \eqref{eq:capac_monotonicity} to conclude that
		\[
			C(n) \log \frac{D_i}{d_i} \leq  \capac(V_i, \intr \R^n \setminus \overline{U_i}) \leq \capac(V_i, \intr V_{i-1}).
		\] 
		We then note that $f - f(x_0) \neq 0$ on $V_{i-1}$ by Lemma \ref{lem:normal_nbhd_props} Moreover, we also have $\abs{f - f(x_0)} = \alpha^{(i-1)/2}\rho$ on $\partial V_{i-1}$ and $\abs{f - f(x_0)} = \alpha^{i/2}\rho$ on $\partial V_{i}$. Thus, we may apply the case $p = n$ of Lemma \ref{lem:basic_adm_function_estimate} with the function $f - f(x_0)$ on the condenser $(V_i, \intr V_{i-1})$. By combining this with Lemma \ref{lem:level_set_lemma} \ref{enum:level_set_integral_over_V} and the definition \eqref{eq:qrval_def} of quasiregular values, we obtain that
		\begin{multline*}
			\capac(V_i, \intr V_{i-1}) \leq \frac{2^n}{\log^n \alpha}\int_{V_{i-1} \setminus V_i} \frac{\abs{Df}^n}{\abs{f-f(x_0)}^n}\\
			\leq \frac{2^n}{\log^n \alpha} \int_{V_{i-1} \setminus V_i} \left( \frac{KJ_f}{\abs{f-f(x_0)}^n} + \Sigma \right)
			=  \frac{2^n}{\log^n \alpha} \int_{V_{i-1} \setminus V_i} \Sigma\\ \leq  \frac{2^n}{\log^n \alpha} \int_{\B^n(x_0, \diam U_f(x_0, \alpha \rho))} \Sigma.
		\end{multline*}
		Our choice made in \eqref{eq:roundness_lemma_setup} now comes into play: chaining it with the previous inequalities yields $\log(D_i/d_i) \leq \log(5/4)$, and consequently
		\begin{equation}\label{eq:chain_part_3}
			d_i \geq \frac{4D_i }{5}
		\end{equation}
		for $i \in \{1, 2\}$. Notably, \eqref{eq:chain_part_3} can be combined with \eqref{eq:chain_part_2}, obtaining
		\begin{equation}\label{eq:chain_combination_1}
			d_1 \geq \frac{8}{5} d_{2}.
		\end{equation}
		
		Finally, we combine our estimates by first applying \eqref{eq:chain_part_1}, then \eqref{eq:chain_combination_1}, and finally \eqref{eq:chain_part_3}. We obtain that
		\[
			2D_2 \geq d_1 + d_2 \geq \left( 1 + \frac{8}{5}\right) d_2
			\geq  \left( 1 + \frac{8}{5}\right) \frac{4}{5} D_2 = \frac{52}{25} D_2 > 2D_2.
		\]
		This is a contradiction as $D_2 > 0$, thus completing the proof.
	\end{proof}
	
	\subsection{Step 2: Logarithmic higher regularity.}
	
	In the second step, we take advantage of Lemma \ref{lem:roundness_lemma} by fitting a component $V'$ of $\R^n \setminus U_f(x_0, \alpha \rho)$ into a cube $\overline{Q}$ with $\sqrt{2}Q \subset V$, and then apply the logarithmic higher integrability result from Corollary \ref{cor:higher_int} to prove Lemma \ref{lem:the_important_distortion_lemma}.
	
	\begin{proof}[Proof of Lemma \ref{lem:the_important_distortion_lemma}]
		For convenience, we denote $D(\rho) = \diam(U_f(x_0, \rho))$ for $\rho > 0$. We start by assuming that $\rho_0$ is small enough that $U_f(x_0, \rho)$ is a strong normal neighborhood of $x_0$ for all $\rho < \rho_0$, and that $\B^n(x_0, D(\alpha^2 \rho_0)) \subset \Omega$, which is guaranteed for small enough $\rho_0$ by Lemma \ref{lem:Ufs_are_strong_normal}. Furthermore, we assume that $\rho_0$ is small enough that
		\begin{equation}\label{eq:dist_lemma_Sigma_assumption}
			\int_{\B^n(x_0, D(\rho_0))} \Sigma  \leq \left( \frac{C_2(n, p)}{2C_1(n, p)} \right)^\frac{n}{p} \log^p \alpha,
		\end{equation}
		where the exponent $p > n$ is as in Corollary \ref{cor:higher_int}, $C_1(n, p)$ is the constant of Corollary \ref{cor:higher_int}, and $C_2(n, p)$ is the constant of Lemma \ref{lem:point_capacity}. As our final condition for $\rho_0$, we assume that it is small enough that 
		\begin{equation}\label{eq:dist_lemma_Sigma_assumption_2}
			\abs{\B^n(x_0, D(\rho_0))}^\frac{p-n}{n} \int_{\B^n(x_0, D(\rho_0))} \Sigma^\frac{p}{n} < \frac{C_2(n, p) \log^p \alpha}{2C_1(n, p)}.
		\end{equation}
		Note that both \eqref{eq:dist_lemma_Sigma_assumption} and \eqref{eq:dist_lemma_Sigma_assumption_2} are possible to assume since the terms on the left hand side tend to $0$ as $\rho_0 \to 0$.
		
		Let $\rho < \rho_0$, and let $V$ be a bounded component of $\R^n \setminus U_f(x_0, \rho)$. We prove that $\abs{f - f(x_0)} \leq \alpha^2 \rho$ on $V$. Suppose towards contradiction that this is not the case, and we can find a $x_1 \in V$ such that $\abs{f(x_1) - f(x_0)} > \alpha^2 \rho$. Then $x_1$ is contained in some bounded component $V'$ of $\R^{n} \setminus U_f(x_0, \alpha \rho)$. By Lemma \ref{lem:roundness_lemma}, we find a ball $B$ such that $V' \subset \overline{B}$ and $2B \subset V$. Thus, we can fix a cube $Q$ such that $V' \subset \overline{Q}$, and $\sqrt{2}Q \subset \intr V$.
		
		By Lemma \ref{lem:point_capacity}, we have
		\begin{equation}\label{eq:point_capacity_chain_1}
			\abs{Q}^{-\frac{p-n}{n}} \leq \frac{\capac_p(\{x_1\}, Q)}{C_2(n, p)}
		\end{equation}
		On the other hand, Lemma \ref{lem:normal_nbhd_props} yields that $f \neq f(x_0)$ on $V$, and we also have $\abs{f - f(x_0)} = \alpha\rho$ on $\partial V'$ and $\abs{f(x_1) - f(x_0)} > \alpha^2\rho$. Thus, using \eqref{eq:capac_monotonicity} and Lemma \ref{lem:basic_adm_function_estimate}, we have
		\begin{equation}\label{eq:point_capacity_chain_2}
			\capac_p(\{x_1\}, Q) \leq \capac_p(\{x_1\}, \intr V') \leq \frac{1}{\log^p \alpha} \int_Q \frac{\abs{Df}^p}{\abs{f-f(x_0)}^p}
		\end{equation}
		We may then apply Corollary \ref{cor:higher_int} on $Q$ with $\intr V$ as the domain of definition. It follows that
		\begin{multline}\label{eq:point_capacity_chain_3}
			\int_Q  \frac{\abs{Df}^p}{\abs{f - f(x_0)}^p}\\
			\leq C_1(n, p) \abs{Q}^{-\frac{p-n}{n}} \left( \int_{\sqrt{2} Q} \frac{\abs{Df}^n}{\abs{f - f(x_0)}^n} \right)^\frac{p}{n} 
			+ C_1(n, p) \int_{\sqrt{2} Q} \Sigma^{\frac{p}{n}}.
		\end{multline}
		Then, applying the fact that $\sqrt{2} Q \subset V$, the definition \eqref{eq:qrval_def} of quasiregular values, Lemma \ref{lem:level_set_lemma} \ref{enum:level_set_integral_over_V}, and finally our assumption \ref{eq:dist_lemma_Sigma_assumption}, we get
		\begin{multline}\label{eq:point_capacity_chain_4}
			\int_{\sqrt{2} Q} \frac{\abs{Df}^n}{\abs{f - f(x_0)}^n} \leq \int_{V} \frac{\abs{Df}^n}{\abs{f - f(x_0)}^n} \leq \int_{V} \left( \frac{KJ_f}{\abs{f - f(x_0)}^n} + \Sigma \right)\\
			= \int_{V} \Sigma \leq \int_{\B^n(x_0, D(\rho_0))} \Sigma 
			\leq \left( \frac{C_2(n, p)}{2C_1(n, p)} \right)^\frac{n}{p} \log^p \alpha
		\end{multline}
	
		By chaining all of \eqref{eq:point_capacity_chain_1}-\eqref{eq:point_capacity_chain_4} together, we get
		\[
			\abs{Q}^{-\frac{p-n}{n}} \leq \frac{\abs{Q}^{-\frac{p-n}{n}}}{2} + \frac{C_1(n, p)}{C_2(n, p) \log^p \alpha} \int_{\sqrt{2} Q} \Sigma^{\frac{p}{n}}.
		\]
		But now, since $\sqrt{2}Q \subset V \subset \B^n(x_0, D(\rho_0))$, we can estimate that
		\[
			\frac{C_2(n, p) \log^p \alpha}{2C_1(n, p)} \leq \abs{Q}^{\frac{p-n}{n}}\int_{\sqrt{2} Q} \Sigma^{\frac{p}{n}} \leq \abs{\B^n(x_0, D(\rho_0))}^\frac{p-n}{n} \int_{\B^n(x_0, D(\rho_0))} \Sigma^\frac{p}{n},
		\]
		which is in direct contradiction with the assumed \eqref{eq:dist_lemma_Sigma_assumption_2}. Thus, there exists no $x_1 \in V$ such that $\abs{f(x_1) - f(x_0)} > \alpha^2 \rho$, completing the proof.
	\end{proof}
	
	\section{Bound on the linear distortion}
	
	Our objective in this section is to prove Theorem \ref{thm:limsup_lin_distortion}. We begin by recalling the statement.
	
	\begin{customthm}{\ref{thm:limsup_lin_distortion}}
		Let $\Omega \subset \R^n$ be a domain, let $x_0 \in \Omega$, and let $f \in W^{1,n}_\loc(\Omega, \R^n)$ be a non-constant continuous map. Suppose that $f$ has a $(K, \Sigma)$-quasiregular value at $f(x_0)$, where $K \geq 1$ and $\Sigma \in L^{1+\eps}_\loc(\Omega)$ with $\eps > 0$. Then
		\[
			\limsup_{r \to 0} \frac{L_f(x_0, r)}{l_f(x_0, r)} \leq C(n, K, i(x_0, f)) < \infty.
		\]
	\end{customthm}

	Our strategy for proving Theorem \ref{thm:limsup_lin_distortion} is to search for ring condensers of the form $(C_1, \R^n \setminus C_2)$, where $C_1 = \overline{U_f(x_0, \rho_1)}$ with $\rho_1$ close to $l_f(x_0, r)$, $C_2$ is the unbounded component of $\R^n \setminus U_f(x_0, \rho_2)$ with $\rho_2$ close to $L_f(x_0, r)$, and both $C_1$ and $C_2$ meet $\partial \B^n(x_0, r)$. Such a ring condenser will allow us to apply Lemma \ref{lem:ring_condenser_capacity} to bound the capacity of this condenser from below, which will lead to a distortion estimate. Corollary \ref{cor:inner_radius_lemma} gives us that $C_1$ meets $\partial \B^n(x_0, r)$ for any choice of $\rho_1 > l_f(x_0, r)$. 
	
	However, $C_2$ is a different story, as while Lemma \ref{lem:inner_and_outer_radius_prelemma} guarantees that $\R^n \setminus U_f(x_0, \rho_2)$ meets $\B^n(x_0, r)$ if $\rho_2 < L_f(x_0, r)$, nothing guarantees that the component of $\R^n \setminus U_f(x_0, \rho_2)$ that meets $\B^n(x_0, r)$ is the unbounded one. This is where our key Lemma \ref{lem:the_important_distortion_lemma} reveals its significance, as it allows us to exclude this possibility by selecting a $\rho_2$ that is smaller by a fixed margin. In particular, the precise result we use is as follows.
	
	\begin{lemma}\label{lem:outer_distortion_lemma_with_QRvar}
		Let $\Omega \subset \R^n$ be a domain, let $x_0 \in \Omega$, let $f \in W^{1,n}_\loc(\Omega, \R^n)$ be a non-constant continuous map, and let $\alpha >  1$. Suppose that $f$ has a $(K, \Sigma)$-quasiregular value at $f(x_0)$, where $K \geq 1$ and $\Sigma \in L^{1+\eps}_\loc(\Omega)$ with $\eps > 0$. Then there exists a $\rho_0 = \rho_0(\Omega, f, x_0, \alpha, K, \Sigma)$ such that, if $\rho \in (0, \rho_0)$ and $\rho < L_f(x_0, r)/\alpha^2$ for some $r > 0$, then the unbounded component $C$ of $\R^n \setminus U_f(x_0, \rho)$ satisfies $C \cap \partial \B^n(x_0, r) \neq \emptyset$.
	\end{lemma}
	\begin{proof}
		We select $\rho_0$ small enough that $U_f(x_0, \rho)$ are normal neighborhoods of $x_0$, and the assumptions of Lemma \ref{lem:the_important_distortion_lemma} apply. In particular, the fact that $U_f(x_0, \rho)$ are normal neighborhoods of $x_0$ assures the existence of a single unbounded component $C$ of $\R^n \setminus U_f(x_0, \rho)$. 
		
		Now, Lemma \ref{lem:the_important_distortion_lemma} assures us that $\abs{f - f(x_0)} \leq \alpha^2 \rho$ on $\R^n \setminus C$. Since $\alpha^2 \rho < L_f(x_0, r)$, there exists a point $x_1 \in \B^n(x_0, r)$ such that $\abs{f(x_1) - f(x_0)} > \alpha^2 \rho$. In particular, we must have $x_1 \in C$, and thus $C \cap \B^n(x_0, r) \neq \emptyset$. Since $C$ is an unbounded connected set that meets $\B^n(x_0, r)$, it follows that $C \cap \partial \B^n(x_0, r) \neq \emptyset$.
	\end{proof}

	We are now ready to prove Theorem \ref{thm:limsup_lin_distortion}. 
	
	\begin{proof}[Proof of Theorem \ref{thm:limsup_lin_distortion}]
		We again note that $f^{-1}\{f(x_0)\}$ is discrete and that $i(x_0, f) > 0$ by Theorem \ref{thm:1p_Reshetnyak}. We fix $\alpha > 1$ and suppose that $\rho_0$ is small enough that $U_f(x_0, \rho)$ is a strong normal neighborhood of $x_0$ and the assumptions of Lemma \ref{lem:outer_distortion_lemma_with_QRvar} are satisfied for $\rho < \rho_0$. By continuity of $f$, we then select $r_0$ small enough that $\B^n(x_0, r_0) \subset U_f(x_0, \rho_0 / 2)$. In particular, for any $r < r_0$, we have $L_f(x_0, r) < \rho_0$.
		
		Our objective is to bound $L_f(x_0, r)/l_f(x_0, r)$ from above for $r < r_0$. Note that $l_f(x_0, r) > 0$ by the openness -part of the single point Reshetnyak's theorem. If we have $L_f(x_0, r)/l_f(x_0, r) \leq \alpha^5$, then we are done. Thus, we focus on the case
		\begin{equation}\label{eq:lin_dist_case}
			\frac{L_f(x_0, r)}{l_f(x_0, r)} > \alpha^5.
		\end{equation} 
		We select $\rho_1 = \alpha l_f(x_0, r)$ and $\rho_2 = L_f(x_0, r)/\alpha^3$, we let $S_1 = \overline{U_f(x_0, \rho_1)}$, and we let $S_2 $ be the unbounded component of $\R^n \setminus U_f(x_0, \rho_2)$. Note that by \eqref{eq:lin_dist_case} and how we selected $\rho_1$ and $\rho_2$, we have 
		\begin{equation}\label{eq:lin_dist_cons_1}
			\frac{\rho_2}{\rho_1} = \frac{1}{\alpha^4} \frac{L_f(x_0, r)}{l_f(x_0, r)} > \alpha.
		\end{equation}
		
		Now, by Corollary \ref{cor:inner_radius_lemma} and Lemma \ref{lem:outer_distortion_lemma_with_QRvar}, both $S_1 $ and $S_2 $ meet $\partial \B^n(x_0, r)$. Moreover, $x_0 \in S_1 $, $S_2 $ is unbounded, and both $S_1 $ and $S_2 $ are connected. Thus, Lemma \ref{lem:ring_condenser_capacity} yields that
		\[
			\capac(S_1 , \R^n \setminus S_2 ) \geq C_1(n) > 0.
		\]
		On the other hand, since $U_{f}(x_0, \rho_1)$ and $U_f(x_0, \rho_2)$ are normal neighborhoods of $x_0$, we have $f(x_0) \notin f(U_f(x_0, \rho_2) \setminus \overline{U_f(x_0, \rho_1)})$. Since also $f(x_0) \notin f V$ for any bounded components $V$ of $\R^n \setminus U_f(x_0, \rho_2)$ by Lemma \ref{lem:normal_nbhd_props}, we have $f(x_0) \notin f(\R^n \setminus (S_1 \cup S_2 ))$. Thus, we may apply Lemma \ref{lem:basic_adm_function_estimate}, the quasiregular value assumption, and Lemma \ref{lem:level_set_lemma} \ref{enum:level_set_integral_over_U} to obtain
		\begin{multline*}
			C_1(n) \leq \capac_n(S_1 , \R^n \setminus S_2 )\\
			\leq \frac{1}{\log^n(\rho_2/\rho_1)}
			\int_{U_f(x_0, \rho_2) \setminus U_f(x_0, \rho_1)} \frac{\abs{Df}^n}{\abs{f - f(x_0)}^n}\\
			\leq \frac{C_2(n) K i(x_0, f)}{\log^{n-1}(\rho_2/\rho_1)} + \frac{1}{\log^n(\rho_2/\rho_1)} \int_{U_f(x_0, \rho_2)} \Sigma.
		\end{multline*}
		Thus, using \eqref{eq:lin_dist_cons_1} and $L_f(x_0, r)/l_f(x_0, r) < 4\alpha^4\rho_2/\rho_1$, we get
		\begin{multline*}
			\frac{1}{C_3(n)} \log^{n-1} \frac{L_f(x_0, r)}{l_f(x_0, r)}\leq \log^{n-1}(4\alpha^4) + \log^{n-1} \frac{\rho_2}{\rho_1}\\
			\leq \log^{n-1}(4\alpha^4) + \frac{C_2(n) K i(x_0, f)}{C_1(n)} + \frac{1}{C_1(n) \log (\rho_2/\rho_1)} \int_{U_f(x_0, \rho_2)} \Sigma\\
			\leq \log^{n-1}(4\alpha^4) + \frac{C_2(n) K i(x_0, f)}{C_1(n)} + \frac{1}{C_1(n)\log(\alpha)} \int_{U_f(x_0, \rho_2)} \Sigma.
		\end{multline*}
		Since $\rho_2 < L_f(x_0, r)$ and $L_f(x_0, r) \to 0$ as $r \to 0$, it follows that the last term in the estimate involving $\Sigma$ tends to zero, and we thus obtain the desired $\limsup$-estimate
		\begin{multline*}
			\limsup_{r \to 0} \frac{L_f(x_0, r)}{l_f(x_0, r)}\\
			 \leq \max \left( \alpha^5, \exp \left(\sqrt[n-1]{ C_3(n) \log^{n-1} 4 \alpha^4 + \frac{ C_3(n) C_2(n) K i(x_0, f)}{C_1(n)}}\right)  \right).
		\end{multline*}
		Note that by letting $\alpha \to 1$ and using $\log^{n-1} 4 \leq (\log^{n-1} 4) K i(x_0, f)$, the above estimate takes a similar form as the corresponding estimate for $K$-quasiregular maps, namely
		\[
			\limsup_{r \to 0} \frac{L_f(x_0, r)}{l_f(x_0, r)} \leq
			\exp \left([C(n) K i(x_0, f)]^\frac{1}{n-1}\right).
		\]
	\end{proof}

	\section{Rescaling principle}
	
	In this section, we prove the rescaling principle for quasiregular values given in Theorem \ref{thm:rescaling}. 
	
	\subsection{Linear distortion and $L^n$-estimates}
	
	We begin the section by pointing out that the linear distortion inequality of Theorem \ref{thm:limsup_lin_distortion} can be connected to $L^n$-norms of the derivative. We only need the lower bound of the following Lemma, but we find the simpler upper bound to also be worth including for completeness.
	
	\begin{lemma}\label{lem:l_Df_L_sandwich}
		Let $\Omega \subset \R^n$ be a domain, let $x_0 \in \Omega$, and let $f \in W^{1,n}_\loc(\Omega, \R^n)$ be a non-constant continuous map. Suppose that $f$ has a $(K, \Sigma)$-quasiregular value at $f(x_0)$, where $K \geq 1$ and $\Sigma \in L^{1+\eps}_\loc(\Omega)$ with $\eps > 0$. Then there exists a constant $C = C(n, K, i(f, x_0))$ such that
		\[
			\limsup_{r \to 0} \frac{L_f(x_0, r)}{\norm{Df}_{L^{n}(\B^n(x_0, r))}} \leq C < \infty.
		\]
		and 
		\[
			\liminf_{r \to 0} \frac{l_f(x_0, r)}{\norm{Df}_{L^{n}(\B^n(x_0, r))}} \geq C^{-1} > 0,
		\]
	\end{lemma}
	\begin{proof}
		We abbreviate $B_r= \B^n(x_0, r)$. We may assume that $r$ is small enough that $U_f(x_0, 2 L_f(x_0, r))$ is a normal neighborhood of $x_0$, and also small enough that 
		\begin{equation}\label{eq:dist_bound_for_small_r_1}
			\frac{L_f(x_0, r)}{l_f(x_0, r)} \leq C_1 = C_1(n, K, i(x_0, f)),
		\end{equation}
		where $C_1/2$ is the upper bound for the limit provided by Theorem \ref{thm:limsup_lin_distortion}. By \eqref{eq:linear_distortion_ball_sandwich}, we have the measure estimate
		\begin{equation}\label{eq:image_volume_sandwich}
			\abs{\B^n} \, l_f^n(x_0, r) \leq \abs{f B_r} \leq \abs{\B^n} L_f^n(x_0, r).
		\end{equation}
		
		We start with the $\limsup$ -estimate, which is simpler. Indeed, we can use \eqref{eq:dist_bound_for_small_r_1}, \eqref{eq:image_volume_sandwich}, and a change of variables to obtain
		\[
			L_f^n(x_0, r)
			\leq C_1^n l_f^n(x_0, r)
			\leq \frac{C_1^n\abs{f B_r}}{\abs{\B^n}}
			\leq \frac{C_1^n}{\abs{\B^n}} \int_{B_r} \abs{J_f}
			\leq \frac{C_1^n}{\abs{\B^n}} \int_{B_r} \abs{Df}^n,
		\]
		which proves the $\limsup$ -estimate.
		
		The $\liminf$ -estimate is slightly more involved, because a change of variables -estimate in this direction requires considering the degree. We start similarly as in the previous case, by estimating that
		\[
			l_f^n(x_0, r) \geq \frac{L_f^n(x_0, r)}{C_1^n} = \frac{1}{2^n C_1^n \abs{\B^n}} \abs{\B^n(f(x_0), 2L_f(x_0, r))}.
		\]
		We then utilize our assumption that $U_f(x_0, 2L_f(x_0, r))$ is a normal neighborhood of $x_0$, where the change of variables result from Lemma \ref{lem:change_of_vars} and the definition of quasiregular values \eqref{eq:qrval_def} yield that
		\begin{multline*}
			\abs{\B^n(f(x_0), 2L_f(x_0, r))} = \frac{1}{i(x_0, f)} \int_{U_f(x_0, 2L_f(x_0, r))} J_f\\
			\geq  \frac{1}{K i(x_0, f)} \int_{U_f(x_0, 2L_f(x_0, r))} \left(\abs{Df}^n - \abs{f - f(x_0)}^n \Sigma \right).
		\end{multline*}
		Now, since $B_r$ is a connected set with $f(B_r) \subset \overline{\B^n}(f(x_0), L_f(x_0, r)) \subset \B^n(f(x_0), 2L_f(x_0, r))$, we have $B_r \subset U_f(x_0, 2L_f(x_0, r))$. Thus, we may estimate
		\begin{multline*}
			\int_{U_f(x_0, 2L_f(x_0, r))} \left(\abs{Df}^n - \abs{f - f(x_0)}^n \Sigma \right)\\
			\geq \int_{B_r} \abs{Df}^n - 2^n L_f^n(x_0, r) \int_{U_f(x_0, 2L_f(x_0, r))} \Sigma\\
			\geq \int_{B_r} \abs{Df}^n - 2^n C_1^n l_f^n(x_0, r) \int_{U_f(x_0, 2L_f(x_0, r))} \Sigma.
		\end{multline*}
		Altogether, after chaining the previous estimates and grouping up all the terms with $l_f^n(x_0, r)$ in them, we get
		\[
			\frac{l_f^n(x_0, r)}{\norm{Df}^n_{L^{n}(\B^n(x_0, r))}}
			\geq \frac{1}{2^n C_1^n} \left( \abs{\B^n} K i(x_0, f) + \int_{U_f(x_0, 2L_f(x_0, r))} \Sigma\right)^{-1}.
		\]
		This proves the $\liminf$ -estimate, as the term with $\Sigma$ tends to 0 as $r \to 0$.
	\end{proof}
	
	\subsection{Construction of the sequence of rescalings}
	
	We then define our choice of rescaled maps that are used throughout the proof of Theorem \ref{thm:rescaling}.
	
	\begin{defn}\label{def:chosen_maps}
		Let $\Omega \subset \R^n$ be a domain and let $x_0 \in \Omega$. Let $f \in W^{1,n}_\loc(\Omega, \R^n)$ be a continuous map with a $(K, \Sigma)$-quasiregular value at $f(x_0)$, where $K \geq 1$ and $\Sigma \in L^{1+\eps}_\loc(\Omega)$ with $\eps > 0$. We denote $y_0 = f(x_0)$. Suppose that $f$ is not the constant function $f \equiv y_0$. 
		
		We select $\rho_0 > 0$ small enough that $U_f(x_0, \rho_0)$ is a normal neighborhood of $x_0$, and we select $r_0 > 0$ small enough that $\B^n(x_0, r_0) \subset U_f(x_0, \rho_0)$. Now, for every $r \in (0, r_0]$, we define a function $g_r \colon \B^n \to \R^n$ by
		\begin{equation}\label{eq:gr_def}
			g_r(x) = f(x_0 + r x) - y_0.
		\end{equation}
		Since $\B^n(x_0, r) \cap f^{-1}\{y_0\} \subset U_f(x_0, \rho) \cap f^{-1}\{y_0\} = \{x_0\}$, we hence have $g_r^{-1}\{0\} = \{0\}$. Therefore, $\norm{g_r}_{L^\infty(\B^n)} > 0$. Thus, we may also define for every $r \in (0, r_0]$ a function $h_r \colon \B^n \to \R^n$ by
		\begin{equation}\label{eq:hr_def}
			h_r(x) = \frac{g_r(x)}{\norm{g_r}_{L^\infty(\B^n)}}.
		\end{equation}
		Notably, every $h_r$ is a continuous map in $W^{1,n}(\B^n, \R^n)$.
	\end{defn}
	
	We first observe that the maps $h_r$ have a $(K, \Sigma_r)$-quasiregular value at the origin, with the weights $\Sigma_r$ getting smaller in some sense as $r \to 0$.
	
	\begin{lemma}\label{lem:hr_qrval_at_origin}
		Under the setting of Definition \ref{def:chosen_maps}, for every $r \in (0, r_0]$, the map $h_r$ has a $(K, \Sigma_r)$-quasiregular value at 0, where
		\[
			\Sigma_r(x) = r^n \Sigma(x_0 + rx)
		\]
		for a.e.\ $x \in \B^n$. In particular, we have
		\[
			\norm{\Sigma_r}_{L^1(\B^n)} = \norm{\Sigma}_{L^1(\B^n(x_0, r))} \xrightarrow[r \to 0]{} 0.
		\]
	\end{lemma}
	\begin{proof}
		We have $g_r = (f \circ \iota_r) - y_0$, where $\iota_r \colon \B^n \to \R^n$ is the affine map $x \mapsto x_0 + rx$. In particular, $\iota_r$ is conformal with $\abs{D\iota_r}^n = J_{\iota_r} \equiv r^n$. Hence, we may compute
		\begin{align*}
			\abs{Dg_r(x)}^n 
			&= \abs{Df(\iota_r(x))}^n J_{\iota_r}(x)\\
			&\leq K J_f(\iota_r(x)) J_{\iota_r}(x) + \abs{f(\iota_r(x)) - y_0}^n \Sigma(\iota_r(x)) J_{\iota_r}(x)\\
			&= K J_{g_r}(x) + \abs{g_r(x)}^n \Sigma_r(x)
		\end{align*}
		for a.e.\ $x \in \B^n$. Thus, $g_r$ has a $(K, \Sigma_r)$-quasiregular value at 0. The same is hence true for $h_r$, since $h_r$ is a scalar multiple of $g_r$. Finally, a standard change of variables shows that
		\[
			\int_{\B^n} \Sigma_r = \int_{\B^n} (\Sigma \circ \iota_r) J_{\iota_r} = \int_{\B^n(x_0, r)} \Sigma,
		\]
		where the right hand side tends to 0 as $r \to 0$ since $\Sigma$ is $L^1_\loc$-integrable.
	\end{proof}
	
	We then use Lemma \ref{lem:l_Df_L_sandwich} to obtain the following estimate for the $L^n$-norm of $Dh_r$.
	
	\begin{lemma}\label{lem:Wn_norm_for_rescaled_maps}
		Under the setting of Definition \ref{def:chosen_maps}, if $r_0$ is small enough, then for every $r \in (0, r_0]$ we have
		\[
			C^{-1} \leq \norm{Dh_r}_{L^n(\B^n)} \leq C,
		\]
		where $C = C(n, K, i(x_0, f))$.
	\end{lemma}
	\begin{proof}
		We again use $\iota_r \colon \B^n \to \R^n$ to denote the map $x \mapsto x_0 + rx$. Then we can compute that
		\[
			\int_{\B^n} \abs{Dh_r}^n = \frac{1}{\norm{g_r}^n_{L^\infty(\B^n)}}\int_{\B^n} (\abs{Df}^n \circ \iota_r) J_{\iota_r} = \frac{\norm{Df}^n_{L^n(\B^n(x_0, r))}}{L_f^n(x_0, r)}.  
		\]
		Thus, given small enough $r_0 > 0$, the lower bound follows immediately from the $\limsup$ -part of Lemma \ref{lem:l_Df_L_sandwich}, and the upper bound follows by combining the $\liminf$ -part of Lemma \ref{lem:l_Df_L_sandwich} with the trivial estimate $L_f(x_0, r) \geq l_f(x_0, r)$.
	\end{proof}
	
	\subsection{Existence of limits and convergence}
	
	We then wish to prove that the family $\{h_r : r \in (0, r_0]\}$ is normal for a small enough $r_0$. This will follow from the Arzel\`a-Ascoli theorem by showing that $\{h_r : r \in (0, r_0]\}$ is totally bounded and locally equicontinuous. Total boundedness is trivial, since $\norm{h_r}_{L^\infty(\B^n)} = 1$ for every $r$. Equicontinuity instead is our first main use of the estimate of Lemma \ref{lem:Wn_norm_for_rescaled_maps} that we derived from the linear distortion bound.
	
	\begin{lemma}\label{lem:local_equicontinuity}
		Under the setting of Definition \ref{def:chosen_maps}, if $r_0 > 0$ is sufficiently small, then the family $\{h_r : r \in (0, r_0]\}$ is locally equicontinuous.
	\end{lemma}
	\begin{proof}
		Let $\eps' \in (0, \min(\eps, 1/(K-1)))$, and let $r_0$ be small enough that Lemma \ref{lem:Wn_norm_for_rescaled_maps} applies. Then for all $r \in (0, r_0)$, $\norm{h_r}_{L^\infty(\B^n)} = 1$ by definition, and $\norm{Dh_r}_{L^n(\B^n)} \leq  C(n, K, i(x_0, f))$ by Lemma \ref{lem:Wn_norm_for_rescaled_maps}. Moreover, we have
		\[
			\int_{\B^n} \Sigma_r^{1+\eps'} = \int_{\B^n} r^{n+n\eps'} \Sigma^{1+\eps'} \circ \iota_r
			\leq r_0^{n\eps'} \int_{\B^n(x_0, r)} \Sigma^{1+\eps'} < \infty,
		\]
		and thus $\norm{\Sigma_r}_{L^{1+\eps'}(\B^n)}$ is bounded independently of $r$. Thus, local equicontinuity follows from the uniform local H\"older continuity estimate given in Lemma \ref{lem:qrval_local_holder}. 
	\end{proof}
	
	With both total boundedness and local equicontinuity shown, the Arzela-Ascoli theorem yields us the following.
	
	\begin{cor}\label{cor:normal_family}
		Under the setting of Definition \ref{def:chosen_maps}, if $r_0 > 0$ is sufficiently small, then the family $\{h_r : r \in (0, r_0]\}$ is normal under local uniform convergence.
	\end{cor}
	
	Thus, for a suitable sequence of rescalings $h_{r_j}$, there exists a limit map $h \colon \B^n \to \R^n$, with $h_{r_j} \to h$ locally uniformly. We then point out that the convergence can be upgraded to weak Sobolev convergence.
	
	\begin{lemma}\label{lem:weak_Sobolev_convergence}
		Let $h_j = h_{r_j}$, where $h_{r}$ are defined according to the setting of Definition \ref{def:chosen_maps}, and $r_j$ is a sequence of radii tending to 0. If $h_j \to h$ locally uniformly, then $h \in W^{1,n}(\B^n, \R^n)$, and a subsequence of $h_j$ converges to $h$ weakly in $W^{1,n}(\B^n, \R^n)$.
	\end{lemma}
	\begin{proof}
		For all $j$, we have $\norm{h_j}_{L^n(\B^n)} \leq \abs{\B^n}^{1/n} \norm{h_j}_{L^\infty(\B^n)}  \leq \abs{\B^n}^{1/n}$, and for large enough $j$, Lemma \ref{lem:Wn_norm_for_rescaled_maps} yields that $\norm{Dh_j}_{L^n(\B^n)} \leq C$. Thus, $h_j$ is a bounded sequence in $W^{1,n}(\B^n, \R^n)$. Moreover, by local uniform convergence, $h_j \to h$ strongly in $L^p(\B^n(0, s), \R^n)$ for every $p \in [1, \infty)$ and $s \in (0, 1)$. Thus, the claim follows from e.g.\ \cite[Proposition VI.7.9]{Rickman_book}.
	\end{proof}
	
	\subsection{Quasiregularity of the limit}
	
	The next statement to prove is that the rescaling limit is $K$-quasiregular. For this, our strategy is similar in spirit to the one presented in e.g.\ \cite[Section VI.8]{Rickman_book}. We begin by pointing out a convergence result for the $L^n$-norms of the approximating functions.
	
	\begin{lemma}\label{lem:Df_convergence}
		Let $h_j, h \in W^{1,n}(\B^n, \R^n)$ be such that $h_j \rightharpoonup h$ weakly in $W^{1,n}(\B^n, \R^n)$. Then for any bounded measurable function $\psi \colon \B^n \to [0, M]$, $M > 0$, we have
		\[
			\int_{\B^n} \psi \abs{Dh}^n \leq \liminf_{j \to \infty} \int_{\B^n} \psi \abs{Dh_j}^n
		\]
	\end{lemma}
	\begin{proof}
		If $h_j \rightharpoonup h$ weakly in $W^{1,n}(\B^n, \R^n)$, then it is easily seen that $Dh_j \rightharpoonup Dh$ weakly in the weighted $L^n$-space $L^n((\B^n, \psi \vol_n), \R^n \times \R^n)$. The claim then follows by a standard argument involving the Hahn-Banach theorem, see e.g.\ \cite[Proposition 2.3.5]{Heinonen-Koskela-Shanmugalingam-Tyson_Book}.
	\end{proof}
	
	We then point out a similar convergence result for the integrals of the Jacobians of the approximating functions.
	
	\begin{lemma}\label{lem:Jf_convergence}
		Let $h_j, h \in W^{1,n}(\B^n, \R^n)$ be such that $h_j \rightharpoonup h$ weakly in $W^{1,n}(\B^n, \R^n)$. Then for any $\phi \in C^\infty_0(\B^n, \R^n)$ 
		\[
			\int_{\B^n} \phi J_h = \lim_{j \to \infty} \int_{\B^n} \phi J_{h_j}.
		\]
	\end{lemma}
	\begin{proof}
		See e.g.\ \cite[(7.9)-(7.11)]{Iwaniec-Martin_book}; there, the result is given for the distributional Jacobian, but naturally the integral of $\phi J_u$ coincides with the image of $\phi$ under the distributional Jacobian of $u$ for $u \in W^{1,n}_\loc(\Omega, \R^n)$.
	\end{proof}
	
	With these two results, we can then show the $K$-quasiregularity of the rescaling limit.
	
	\begin{lemma}\label{lem:limit_is_QR}
		Let $h_j = h_{r_j}$, where $h_{r}$ are defined according to the setting of Definition \ref{def:chosen_maps}, and $r_j$ is a sequence of radii tending to 0. If $h_j \to h$ locally uniformly, then $h$ is $K$-quasiregular.
	\end{lemma}
	\begin{proof}
		By moving to a subsequence, we may assume by Lemma \ref{lem:weak_Sobolev_convergence} that $h_j \rightharpoonup h$ weakly in $W^{1,n}(\B^n, \R^n)$. Let $\phi \in C^\infty_0(\B^n)$ with $\phi \geq 0$ everywhere. We use Lemmas \ref{lem:Df_convergence}, \ref{lem:Jf_convergence}, and \ref{lem:hr_qrval_at_origin} along with the fact that $\abs{h_j} \leq 1$ to conclude that
		\begin{multline*}
			\int_{\B^n} \left( \abs{Dh}^n - K J_h \right) \phi
			\leq \liminf_{j \to \infty } \int_{\B^n} \left( \abs{Dh_j}^n - K J_{h_j} \right) \phi\\
			\leq \liminf_{j \to \infty } \int_{\B^n} \phi \abs{h_j}^n \Sigma_{r_j}
			\leq \liminf_{j \to \infty} \norm{\phi}_{L^\infty(\B^n)}  \norm{\Sigma_{r_j}}_{L^1(\B^n)} = 0.
		\end{multline*}
		
		It follows from this that $\abs{Dh(x)}^n - K J_h(x) \leq 0$ for a.e.\ $x \in \R^n$. Indeed, if $B$ is a ball such that $2B \subset \B^n$, then we may select $\phi_j \in C^\infty_0(\B^n, [0, 1])$ such that $\phi_j \equiv 1$ on $B$ and $\spt \phi_j \subset (1+j^{-1})B$. By applying the previous estimate for these test functions $\phi_j$, it follows that
		\[
			\int_{B} \left( \abs{Dh}^n - K J_h \right) \leq \liminf_{j \to 0}  \int_{[(1+j^{-1})B] \setminus B} \abs{K J_h - \abs{Dh}^n} = 0.
		\]
		Since this holds for all balls $B$ with $2B \subset \B^n$, we obtain that $\abs{Dh(x)}^n - K J_h(x) \leq 0$ for a.e.\ $x \in \B^n$ by the Lebesgue differentiation theorem.
	\end{proof}
	
	\subsection{Non-degeneracy of the limit}
	
	We recall the statement of Theorem \ref{thm:rescaling}.
	
	\begin{customthm}{\ref{thm:rescaling}}
		Let $\Omega \subset \R^n$ be a domain and  $x_0 \in \Omega$. Suppose that $f \colon \Omega \to \R^n$ is a non-constant continuous map having a $(K, \Sigma)$-quasiregular value at $f(x_0)$, where $K \geq 1$ and $\Sigma \in L^{1+\eps}_\loc(\Omega)$ with $\eps > 0$. Then there exist radii $r_j > 0$ and scaling factors $c_j > 0$, $j \in \Z_{> 0}$, such that the functions $h_j \colon \B^n \to \R^n$ defined by
		\[
			h_j(x) = c_j(f(x_0 + r_j x) - f(x_0))
		\]
		converge both locally uniformly and weakly in $W^{1,n}(\B^n, \R^n)$ to a non-constant $K$-quasiregular map $h \colon \B^n \to \R^n$.
	\end{customthm}
	
	Almost all the stated properties for the rescaling process have already been proven. However, one last absolutely crucial detail remains, which is to show that the obtained limit map is non-constant; indeed, without this detail, the statement of Theorem \ref{thm:rescaling} would be trivial, since any continuous map can be rescaled in a way that makes it converge locally uniformly to a constant map. We hence complete the proof of Theorem \ref{thm:rescaling} by showing this fact.
	
	\begin{proof}[Proof of Theorem \ref{thm:rescaling}]
		We let $g_r$ and $h_r$, $r \in (0, r_0]$, be as in the Definition \ref{def:chosen_maps}. We choose our sequence of $r_j$ strategically: we use Lemma \ref{lem:Ufs_are_strong_normal} to select $\rho_j \to 0$ small enough that $U_f(x_0, \rho_j)$ are strong normal neighborhoods of $x_0$, and then select $r_j = l_f^*(x_0, \rho_j)$. Since $f^{-1}\{f(x_0)\}$ is discrete, we have $r_j \to 0$. Moreover, our strategic selection ensures that by Lemma \ref{lem:lin_inv_dil_cancel}, we have $\smallnorm{g_{r_j}}_{L^\infty(\B^n)} = L_f(x_0, r_j) = \rho_j$.
		
		Let then $h_j = h_{r_j}$. By Corollary \ref{cor:normal_family}, after moving to a subsequence, we have $h_j \to h$ locally uniformly in $\B^n$. By Lemma \ref{lem:weak_Sobolev_convergence}, after moving to a further subsequence, we have $h_j \rightharpoonup h$ weakly in $W^{1,n}(\B^n, \R^n)$. Now, with Lemma \ref{lem:limit_is_QR}, we obtain that $h$ is $K$-quasiregular. It remains to show that $h$ is non-constant.
		
		For this, we first consider the case that, for a given $s \in (0, 1)$, we have $h_j \B^n(0, s) \setminus \B^n(0, 1/4) \neq \emptyset$ for infinitely many $j$. Then, by the uniform convergence of $h_j$ to $h$ on $\overline{\B^n}(0, s)$ and the compactness of $\overline{\B^n}(0, s)$, we find a point $x \in \overline{\B^n}(0, s)$ such that $\abs{h(x)} \geq 1/4$. Since $h(0) = 0$, it follows that in this case, $h$ is non-constant.
		
		Therefore, we assume towards contradiction that for every $s \in (0, 1)$, we have $h_j \B^n(0, s) \subset \B^n(0, 1/4)$ for all but finitely many $j$. We fix $s$ and denote $U_j^1 = U_f(x_0, \rho_j/2)$, and $U_j^2 = U_f(x_0, \rho_j/4)$ for brevity, with $\overline{U_j^2} \subset U_j^1$; see Figure \ref{fig:non-const_setup} for an illustration. By moving to a subsequence in $j$, we may assume that $h_j (\B^n(0, s)) \subset \B^n(0, 1/4)$ for every $j$.  This in turn implies that $f \B^n(x_0, s r_j) \subset \B^n(f(x_0), \rho_j/4)$, and consequently, since $\B^n(x_0, s r_j)$ is connected, we have $\B^n(x_0, s r_j) \subset U_j^2$ for every $j$.
		
		\begin{figure}[h]
			\begin{tikzpicture}[scale = 1]
				\draw[fill=lightgray!20]
				plot [smooth cycle,tension=0.5] 
				coordinates {
					(-2.6, 0) (-2.4, 2.7) (0, 2.6) (2, 2) 
					(2.4, 0.3) (2.3, -1.6) (0, -2.4) (-2.3, -2.5)
				};
				\draw [fill=white] 
				plot [smooth cycle,tension=0.5] 
				coordinates{
					(-1.20, -2.1) (-1.20, -1.7) (-0.90, -1.8) (-1.00, -2.1)
				};
				\draw [fill=white] 
				plot [smooth cycle,tension=0.5] 
				coordinates{
					(-1.8, 2.3) (-1.8, 1.9) (-1.4, 2.0) (-1.5, 2.3)
				};
				
				\draw[fill=lightgray!50]
				plot [smooth cycle,tension=0.5] 
				coordinates {
					(-2.2, 0) (-2, 1.5) (0, 2.2) (2, 1.3) 
					(2.2, 0) (2, -1.4) (0, -2.2)  (-1, -1.2) (-2, -1.7)
				};
				\draw [fill=white] 
				plot [smooth cycle,tension=0.5] 
				coordinates{
					(-1.3, 1.4) (-1.4, 1.2) (-0.9, 1.1) (-0.8, 1.4)
				};
				\draw [fill=lightgray!20]
				plot [smooth cycle,tension=0.5] 
				coordinates{
					(-1.3, 1.4) (-1.4, 1.2) (-0.9, 1.1) (-0.8, 1.4)
				};
				
				\draw[fill=gray!50] 
				plot [smooth cycle,tension=0.5] 
				coordinates {
					(-0.9, 0) (-0.7, 1) (0.9, 0.8) (0.82, 0) (0.9, -0.9) (-0.8, -0.8)
				};

				\draw (0.2,0) node[anchor=south east] {$U_j^2$};
				\draw (-1.4,0) node[anchor=center] {$U_j^1$};

				\draw[dashed] (0,0) circle (0.8);
				\draw[dashed] (0,0) circle (2);
				
				\draw (0,0) -- (1.2,0) node[anchor=south] {$r_j$} -- (2, 0);
				\draw (0,0) -- (0.57,-0.57) node[anchor=east] {$sr_j$};

				\draw[fill=lightgray!20] (7.5,0) circle (3);
				\draw[fill=lightgray!50] (7.5,0) circle (1.5);
				\draw[fill=gray!50] (7.5,0) circle (0.75);
				
				\draw (7.5,0) -- (8.25,0);
				\draw (8.25, -0.1) node[anchor=south east] {$\rho_j/4$};
				\draw (7.5,0) -- (7.5,-1.5);
				\draw (7.5, -1.4) node[anchor=south west] {$\rho_j/2$};
				\draw (7.5,0) -- (7.5-2.6,-1.5);
				\draw (7.5-2.1, -1.25) node[anchor=south] {$\rho_j$};
				
				\draw[->] (3,0) -- (3.5,0) node[anchor=south] {$f$} -- (4, 0);		
			\end{tikzpicture}
			\caption{Example illustration of various objects in the proof of Theorem \ref{thm:rescaling}. The counterassumption $h_j \B^n(0, s) \subset \B^n(0, 1/4)$ implies that $U_j^2$ contains the inner circle, while $L_f(x_0, r_j) = \rho_j$ and Lemma \ref{lem:outer_distortion_lemma_with_QRvar} imply that the unbounded component of $\R^n \setminus U_j^1$ meets the outer circle.}\label{fig:non-const_setup}
		\end{figure}
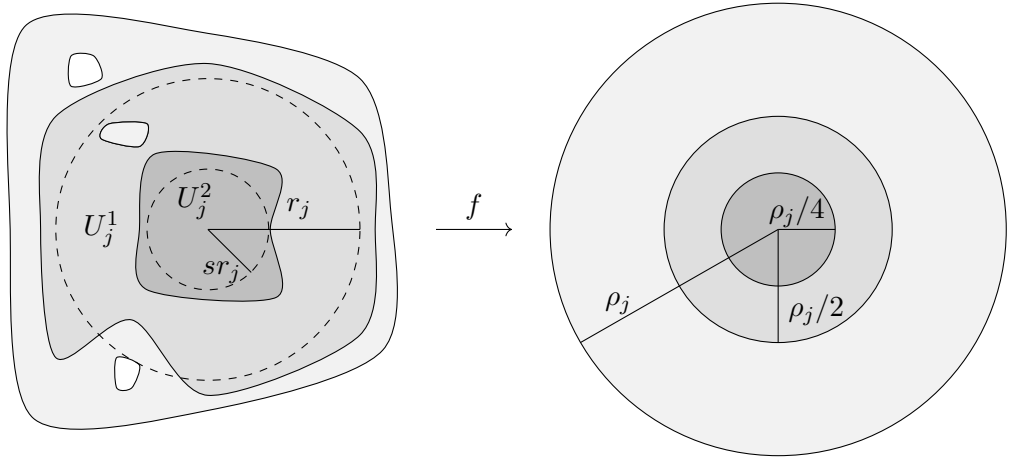
		
		On the other hand, since $L_f(x_0, r_j) = \rho_j$, an application of Lemma \ref{lem:outer_distortion_lemma_with_QRvar} with $\alpha = \sqrt[3]{2}$ yields that, after moving to a further subsequence, we may assume that $\partial \B^n(x_0, r_j)$ meets the unbounded component $C^1_j$ of $\R^n \setminus U^1_j$ for every $j$. Thus, $\diam(U_j^2) \geq \diam(\B^n(x_0, s r_j)) = 2 s r_j$, $\diam (C_j^1) = \infty$, and $\dist(U_j^2, C_j^1) \leq \dist(\B^n(x_0, s r_j), C_j^1 \cap \partial \B^n(x_0, r_j)) = (1-s)r_j$. We may thus use the Loewner property from Lemma \ref{lem:Loewner} and the monotonicity of capacity from \eqref{eq:capac_monotonicity} to conclude that
		\begin{multline}\label{eq:capacity_lower_bound}
			\capac(\overline{U_j^2}, U_j^1)
			\geq \capac(\overline{U_j^2}, \R^n \setminus C^1_j)\\ 
			\geq C(n) \log\frac{2s r_j}{(1-s) r_j} = C(n) \log\frac{2s}{1-s}
		\end{multline}
		for all $j$. Note that this lower capacity bound is independent of $j$, and tends to $\infty$ as $s \to 1$.
		
		But now, we may again apply Lemma \ref{lem:basic_adm_function_estimate} to the condenser $(\overline{U_j^2}, U_j^1)$. Combining this with \eqref{eq:qrval_def} and Lemma \ref{lem:level_set_lemma} \ref{enum:level_set_integral_over_U}, we obtain that
		\begin{multline*}
			\capac(\overline{U_j^2}, U_j^1)
			\leq \frac{1}{\log^n 2} \int_{U_j^1 \setminus U_j^2} \frac{\abs{Df}^n}{\abs{f - f(x_0)}^n}\\
			\leq \frac{1}{\log^n 2} \int_{U_j^1 \setminus U_j^2} \left( \frac{K J_f}{\abs{f - f(x_0)}^n} - \Sigma \right)
			\leq \frac{C(n) K i(x_0, f)}{\log^{n-1} 2} + \frac{1}{\log^{n} 2}\int_{U_j} \Sigma.
		\end{multline*}
		This upper bound in particular is independent of $s$, and tends to a finite value of $C(n) K i(x_0, f) \log^{-(n-1)}(2)$ as $j \to \infty$. Thus, by selecting $s$ close enough to 1 and a sufficiently large $j$, we obtain a contradiction with \eqref{eq:capacity_lower_bound}, completing the proof. 
	\end{proof}
	
	\section{Proof of the small $K$ -theorem}
	
	In this very brief section, we prove Theorem \ref{thm:small_K}. We again first recall the statement.
	
	\begin{customthm}{\ref{thm:small_K}}
		Let $\Omega \subset \R^n$ be a domain, let $x_0 \in \Omega$, and let $f \in W^{1,n}_\loc(\Omega, \R^n)$ be a non-constant continuous map. Suppose that $f$ has a $(K, \Sigma)$-quasiregular value at $f(x_0)$, where $\Sigma \in L^{1+\eps}_\loc(\Omega)$ with $\eps > 0$. If $K < K_0(n)$, where $K_0(n)$ is Rajala's constant, then $i(x_0, f) = 1$.
	\end{customthm}
	
	\begin{proof}
		Let $h_j$ be the rescalings of $f$ from Theorem \ref{thm:rescaling}, in which case $h_j$ converge to a non-constant $K$-quasiregular $h$ locally uniformly in $\B^n$. By the topological invariance of the local index, we see that $i(0, h_j) = i(x_0, f)$ for every $j$. On the other hand, by the small $K$ -theorem, $h$ is an orientation-preserving local homeomorphism, and thus $i(0, h) = 1$.
		
		Since $h$ is quasiregular, we may select $\rho \in (0, 1)$ such that $U_h(0, \rho)$ is a normal neighborhood of $0$, and $r \in (0, 1)$ such that $\B^n(0, r) \subset U_h(0, \rho)$. In particular, since $h^{-1}\{0\} \cap U_h(0, \rho) = \{0\}$, we obtain that the distance $D = \dist(0, h \partial \B^n(0, r))$ is positive. We select $j$ large enough that $\B^n(x_0, r_j)$ is contained in a normal neighborhood $U_0$ of $x_0$ with respect to $f$, and also large enough that local uniform convergence yields us $\abs{h_j - h} < D/2$ uniformly on $\partial \B^n(0, r)$.
		
		Now, we can define a line homotopy $H \colon \overline{\B^n}(0, r) \times [0,1] \to \R^n$ from $h$ to $h_j$, and by the fact that $\abs{h_j - h} < D/2$ on $\partial \B^n(0, r)$, we see that $0 \notin H(\partial \B^n(0, r) \times [0,1])$. Thus, by the homotopy invariance of the degree given in Lemma \ref{lem:top_degree_props} \eqref{enum:deg_homotopy}, we have
		\[
			\deg(h, 0, \B^n(0, r)) = \deg(h_j, 0, \B^n(0, r)).
		\]
		However, we may now use the local index summation formula from Lemma \ref{lem:local_index_summation} to conclude that
		\[
			i(x_0, f) = i(0, h_j) = \deg(h_j, 0, \B^n(0, r)) = \deg(h, 0, \B^n(0, r)) = i(0, h) = 1,
		\]
		completing the proof.
	\end{proof}
	

\end{document}